\documentclass[11pt,oneside]{amsart}

\usepackage[utf8]{inputenc}
\usepackage[T1]{fontenc}
\usepackage{lmodern}
\usepackage[margin = 3cm]{geometry}
\usepackage{array}
\usepackage{fancyhdr}
\usepackage{amsmath}
\usepackage{amsfonts}
\usepackage{amssymb}
\usepackage{graphicx}
\usepackage{float}
\usepackage{caption}
\usepackage{pgf,tikz}
\usetikzlibrary{arrows}
\usepackage{tikzlings-addons}
\usepackage{tikzlings-penguins}[2021/08/06 version v0.8 Draw penguins in TikZ]
\usetikzlibrary{svg.path}
\usepackage{subcaption}
\usepackage{parskip}
\usepackage[export]{adjustbox}

\usepackage{verbatim}
\usepackage{hyperref}

\definecolor{ffttww}{rgb}{1,0.2,0.4}
\definecolor{uuuuuu}{rgb}{0.27,0.27,0.27}
\definecolor{cqcqcq}{rgb}{0.75,0.75,0.75}
\definecolor{wwqqcc}{rgb}{0.4,0,0.8}

\newcommand{\R}{\mathbb{R}}
\newcommand{\N}{\mathbb{N}}
\newcommand{\HH}{\mathcal{H}}
\newcommand{\RR}{\mathcal{R}}
\newcommand{\A}{\mathcal{A}}
\newcommand{\B}{\mathcal{B}}
\newcommand{\est}{\emptyset}

\usepackage{tikzlings}
\usepackage{tikzducks}
\newtheorem{theorem}{Theorem}

\newtheorem{corollary}{Corollary}
\newtheorem{lemma}{Lemma}
\newtheorem{definition}{Definition}

\newtheorem{conj}{Conjecture}

\let\svthefootnote\thefootnote
\newcommand\freefootnote[1]{%
  \let\thefootnote\relax%
  \footnotetext{#1}%
  \let\thefootnote\svthefootnote%
}

\title[Hitting sets and colorings of hypergraphs]%
  {Hitting sets and colorings of hypergraphs}

\author{Bal\'azs Bursics}
\email{bursicsb@student.elte.hu \vspace*{-0.5cm}}
\address{Eötvös Loránd University Faculty of Science, Pázmány Péter sétány 1/A, H-1117 Budapest, Hungary}

\author{Bence Csonka}
\email{csonkab@edu.bme.hu \vspace*{-0.5cm}}
\address{Budapest University of Technology and Economics, M\H{u}egyetem rkp. 3., H-1111 Budapest, Hungary}

\author{Luca Szepessy}
\email{szepessyluca@student.elte.hu \vspace*{-0.5cm}}
\address{Eötvös Loránd University Faculty of Science, Pázmány Péter sétány 1/A, H-1117 Budapest, Hungary}

\begin{document}

\maketitle

\thispagestyle{empty}

\begin{abstract}	

In this paper we study the minimal size of edges in hypergraph families that guarantees the existence of a polychromatic coloring, that is, a $k$-coloring of a vertex set such that every hyperedge contains a vertex of all $k$ color classes. We also investigate the connection of this problem with $c$-shallow hitting sets: sets of vertices that intersect each hyperedge in at least one and at most $c$ vertices.

We determine for some hypergraph families the minimal $c$ for which a $c$-shallow hitting set exists.

We also study this problem for a special hypergraph family, which is induced by arithmetic progressions with a difference from a given set. We show connections between some geometric hypergraph families and the latter, and prove relations between the set of differences and polychromatic colorability.

\end{abstract}

\section{Introduction}


\freefootnote{MSC2020: 05C15, 05C65. Key words and phrases: geometric hypergraphs, polychromatic coloring, shallow hitting
sets, arithmetic progressions}

A {\it polychromatic k-coloring} of a hypergraph $H$ is a $k$-coloring of its vertex set such that every hyperedge contains a vertex of all $k$ color classes. Observe that a polychromatic 2-coloring is the same as the usual proper 2-coloring of hypergraphs, where we require that no edge is monochromatic. By merging $j$ color classes of a polychromatic $k$-coloring we get a polychromatic $(k-j+1)$-coloring, so the condition of monochromatic $k$-colorability becomes stricter as $k$ increases. A trivial necessary condition for the existence of a polychromatic $k$-coloring is that all edges of $H$ must be of size at least $k$.

\begin{definition} Denote by $H_{\ge m}$ or simply $H_m$ the hypergraph obtained from $H$ by deleting all hyperedges of size smaller than $m$, and denote by $H_{=m}$ the hypergraph consisting of the hyperedges of $H$ with size exactly $m$. Similarly, for a hypergraph family $\mathcal{H}$ let $\HH_m=\HH_{\ge m}=\{H_{\ge m}:H\in\HH\}$ and $\HH_{=m}=\{H_{=m}:H\in\HH\}$.
\end{definition}

One can make statements about the existence of polychromatic $k$-colorings for every member of a hypergraph family $\HH$ using the following parameter:

\begin{definition} Let $\mathcal{H}$ be a hypergraph family. Denote by $m_\mathcal{H}(k)$ the smallest positive integer $m$ such that for every $H\in\mathcal{H}$ there exists a polychromatic $k$-coloring of the hypergraph $H_{\ge m}$. If there is no such $m$, set $m_{\mathcal{H}}(k)=\infty$.
\end{definition}

Determining or bounding $m_{\mathcal{H}}(k)$ is an interesting problem in itself for some hypergraph families. {\it Range capturing hypergraph families} are particularly well studied: Suppose $\mathcal{S}$ is a family of planar (or higher dimensional) sets, called the {\it range family}, consider the family of hypergraphs whose members have a finite vertex set $V$, and whose edge set consists of all subsets $e\subseteq V$ such that $e=V\cap S$ for some $S\in\mathcal{S}$, in which case we say that the range $S$ captures $e$. Much research has been done on polychromatic colorings and the parameter $m(k)$ of such hypergraph families, for example, when the range family consists of halfplanes \cite{felsik}, translates of a polygon \cite{varadarajan}, translates of a convex body \cite{damasdi2021three}, homothets of a polygon \cite{damasdi2022realizing}, translates of an octant \cite{ternyolcad}, axis-parallel rectangles \cite{teglalap}, or axis-parallel strips \cite{apstrip}. For a comprehensive summary, see the website \cite{cogezoo} maintained by Keszegh and Pálvölgyi.

Investigating this problem is also motivated by the connection between polychromatic colorings and the cover-decomposability of planar sets. Suppose that we have a planar polygon $P$, and some translates of $P$ are given in such a way that every point of the plane is covered at least $m$ times. A natural question is whether this cover can be decomposed into two sets such that both set of translates of $P$ covers the whole plane in itself. The method of dualization can be used to reduce this to the problem of polychromatic colorability where the vertex set is the set of centers of gravity of the translate polygons, and the hyperedges are the sets contained in any translate of $P$. This is described in detail in \cite{dekompozicio}, for some related results see e.g. \cite{pach2016unsplittable, palvolgyi2010convex, palvolgyi2010indecomposable, varadarajan, keszegh2012octants, kovacs2015indecomposable}.

Now we turn to shallow hitting sets, and their role in constructing polychromatic colorings under suitable conditions.

\begin{definition}
 Let $H=(V,E)$ be a hypergraph, a $U\subset V$ vertex set $U\subset V$ is a {\it c-shallow hitting set} for a positive integer $c$ if for every hyperedge $e\in E$: $1\le|e\cap U|\le c.$
\end{definition} 

\begin{definition}
 Let $H=(V,E)$ be a hypergraph, and $V'\subset V$. Define $E'=\{e'\subset V'| \exists e\in E: e'=e\cap E'\}$, we say that $E'=(V',E')$ is the induced subhypergraph of $H$ on $V'$.
\end{definition}

Suppose that $\HH$ is a hypergraph family such that for arbitrary $H=(V,E)\in\HH$ and $V'\subset V$ the induced subhypergraph $H'$ is a member of $\HH$. Then $c$-shallow hitting sets can be used to create polychromatic colorings, see e.g. \cite{felsik, pshalfplane}.

In these papers $c$-shallow hitting sets are applied through the implicit use of the following lemma, which we restate here (and which also follows from \cite[Theorem 2.11.]{pshalfplane}):

\begin{lemma}\label{prop1}
Suppose that the hypergraph family $\mathcal{H}$ satisfies the condition that for arbitrary $H=(V,E)\in\HH$ and $V'\subset V$, the induced subhypergraph of $H$ on $V'$ is an element of $\HH$, and that every hyperedge $e\in E$ contains a vertex $v\in e$ such that for $e'=e\setminus \{v\}$ and $E'=E\setminus \{e\}\cup\{e'\}$ we have $H'=(V,E')\in\HH.$
Suppose that for every $m\ge c$ and every element of $\mathcal{H}_{=m}$ there exists a $c$-shallow hitting set. Then $m_\mathcal{H}(k) \leq c \cdot (k-1) +1$.
\end{lemma}

For range capturing hypergraph families, the first condition of the above statement -- that taking an induced subhypergraph does not lead out of the family -- suffices automatically, because if a range $S$ captures an edge $e\subset V$, then the same $S$ also captures $e\cap V'$.

On the other hand, the existence of $c$-shallow hitting sets is not so evident, for example it is a frequent case for geometric hypergraph families that a hyperedge contains more than $c$ pairwise disjoint hyperedges. 
Thus we need to restrict to a {\it Sperner} subfamily of the original hypergraph family (that is, where no two edges are contained one in another), to have a chance for the existence of $c$-shallow hitting sets. However, in some cases, e.g., hypergraphs induced by bottomless rectangles (see below), this is not strong enough either \cite{pshalfplane}. 

Fortunately, Lemma \ref{prop1}. is about a special type of Sperner subfamily, the $m$-uniform members of the family, which gives even better chances to obtain $c$-shallow hitting sets. Therefore, for a given hypergraph family $\HH$, it is natural to ask whether there is a positive integer $c$ such that there exists a $c$-shallow hitting set for all $H_{=m}\in \HH_{=m}$.



Since the resulting bound is linear in $k$, Lemma \ref{prop1}. is also loosely related to an important conjecture of the field:

\begin{conj}[\cite{dekompozicio}]
If $m_\HH(2)<\infty$ for a hypergraph family $\HH$, then $m_\HH(k)=O(k)$.
\end{conj}

It would also be interesting to see whether Lemma \ref{prop1}. is reversible, more precisely, does an upper bound of $m_\mathcal{H}(k)$, which is linear in $k$, guarantee a constant $c$ independent from $k$ such that there exists a $c$-shallow hitting set for every member of $\HH_{=m}$?

\subsection{Bottomless rectangles}

\begin{definition} Denote by $\mathcal{B}$ the hypergraph family which consists of those hypergraphs $H=(V,E)$, for which $V\subset \R^2$ is finite, and the edges are induced by bottomless rectangles: every $e\in E$ edge can be written in the form $e=\{(x,y)\in V: x_0 < x < x_1, y<y_0\}$ for some $x_0, x_1, y_0 \in\R$. 
\end{definition}
\begin{figure} [h!]
    \centering
    \includegraphics[width=0.7\textwidth]{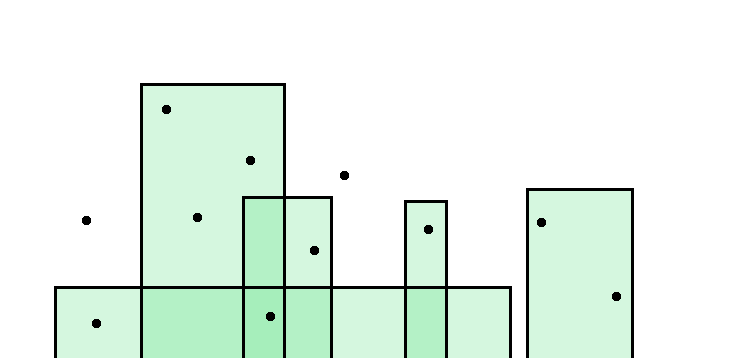}
    \caption{Hypergraph induced by bottomless rectangles}
\end{figure}

It is already known \cite{feneketlen} that $1.67k\le m_{\mathcal{B}}(k)\le 3k-2$, and that for an arbitrary $c$ there is a Sperner member of the family with no $c$-shallow hitting set \cite{pshalfplane}. The existence of 2-shallow hitting sets on $\B_{=m}$ would imply $m_{\mathcal{B}}(k)\le 2k-1$, and the existence of 3-shallow hitting sets would give another proof of the upper bound of $3k-2$. Keszegh and Pálvölgyi \cite{pshalfplane}, and also Chekan and Ueckerdt \cite{chekan2022polychromatic} asked whether this could be the case. However, our following result refutes these possibilities:

\begin{theorem}\label{bottomless}
Let $m \ge 12$, then there is a member of $\mathcal{B}_{=m}$ which does not have a 3-shallow hitting set.
\end{theorem}

We do not yet know whether there is a 4-shallow hitting set for any $\mathcal{B}_{=m}$ with $m$ large enough. However, Planken and Ueckerdt recently showed that there is a 10-shallow hitting set for any member of $\mathcal{B}_{=m}$ \cite{bless10}.

\subsection{Axis-parallel strips}

We also investigated this question on another geometric hypergraph family and its dual:

\begin{definition} Denote by $\A$ the hypergraph family which consists of those hypergraphs $H=(V,E)$ for which $V\subset \R^2$ is finite, and the edges are induced by axis-parallel strips: every $e\in E$ edge can be presented in the form $V\cap\{(x,y)\in\R^2: x_0<x<x_1\}$ or $V\cap\{(x,y)\in\R^2: y_0<y<y_1\}$ for some $x_0, x_1, y_0, y_1 \in\R$. Denote by $\A^*$ its dual: the family of such hypergraphs that the vertices are axis-parallel strips, and the edges are sets consisting of the strips containing a given point $(x,y)$. 
\end{definition}

The following results are presented in \cite{apstrip}: 
$$2\bigg\lceil\frac{3k}{4}\bigg\rceil-1 \le m_\A(k)\le 2k-1$$
$$2\bigg\lceil \frac{k}{2}\bigg\rceil +1\le m_{\A^*}(k)\le 2k-1$$

The lower bound for $m_{\A^*}(k)$ comes from the more general case of arbitrary dimension, however, in the plane we can improve on this bound:

\begin{theorem}\label{strip_dual}
For the hypergraph family $\A^*$ we have 
$2\left(\left\lceil \frac{3}{4}k\right\rceil -1\right) + 1 \leq m_{\A^*}(k)$.
\end{theorem}

The existence of 2-shallow hitting sets on the hypergraph family $\A$ would give an alternative proof for the upper bound of $m_\A(k)$ via Lemma \ref{prop1} but this is not the case for this problem either.

\begin{theorem}\label{strip}
For sufficiently large $m$ there is an element in $\A_{=m}$ with no 2-shallow hitting set.
\end{theorem}

Note that for given $m$ and for $k= \lceil \frac{m}{2}\rceil$ colors, taking a color class in the coloring corresponding to the upper bound in \cite{apstrip} gives a 3-shallow hitting set, so we have determined the smallest possible $c$ with a $c$-shallow hitting set for every member of $\A_{=m}$. 


This problem motivated the study of the following more general hypergraph families:

\begin{definition} Denote by $\A^+$ the hypergraph family which consists of such hypergraphs $H=(V,E)$ for which $V\subset \R^2$ is finite, and the edges are induced by the union of a horizontal and a vertical axis-parallel strip: every $e\in E$ edge can be presented in the form $V\cap\big(\{(x,y)\in\R^2: y_0<y<y_1\}\cup \{(x,y)\in\R^2: x_0<x<x_1\}\big)$ for some $x_0,x_1,y_0,y_1\in\R$. 
\end{definition}

\begin{definition} Denote by $\A_s$ the hypergraph family which consists of such $H=(V,E)$ hypergraphs for which $V\subset \R^2$ is finite, and the edges are induced by the union of $s$ axis-parallel strips: every $e\in E$ edge can be written in the form $V\cap \Big(\cup_{i=1}^s A_i\Big)$ where $A_1,\ldots, A_s$ are axis-parallel strips. 
\end{definition}

We have the following bounds about their polychromatic colorings:

\begin{theorem}\label{kereszt}
$$3\bigg\lceil \frac{3}{4}k\bigg\rceil -2\le m_{\A^+}(k)\le 4k-3.$$
\end{theorem}

We note that in the case of two colors, Theorem \ref{kereszt} yields $4\le m_{\A^+}(2)\le5$. However, a simple construction shows that in fact the upper bound is sharp, that is, $m_{\A^+}(2)=5$.

\begin{theorem}\label{s_strip}
$$m_{\A_s}(k)=s \cdot m_\A(k)-s+1.$$
\end{theorem}

Combining Theorem \ref{s_strip} with the bounds from \cite{apstrip} imply the following:

\begin{corollary}
$2s(\lceil \frac{3}{4} k\rceil-1)+1\le m_{\A_s}(k)\le s(2k-2)+1$, in particular $m_{\A_s}(2)=2s+1$, and $m_{\A_s}(3)=4s+1.$
\end{corollary}

\subsection{Arithmetic progressions}

A famous result of van der Waerden states that for any finite coloring of $\mathbb{N}$ there exists an arbitrarily long monochromatic arithmetic progression, or in other words, if we take the hypergraph on $\mathbb{N}$ with finite arithmetic progressions as edges, then every finite coloring contains a monochromatic hyperedge of arbitrary size.

There are many well-studied related problems, one of them being how does a {\it ladder}, a set $L\subseteq\mathbb{N}$ such that all finite colorings of $\mathbb{N}$ contain arbitrarily long arithmetic progressions with difference $d\in L$ look like \cite{ladder1, ladder2}. Translated again to the hypergraph terminology, let $H^L$ be the hypergraph obtained from the previously mentioned hypergraph by keeping only those edges which are derived from arithmetic progressions with difference from some subset $L\subseteq \mathbb{N}$,  then $L$ is a ladder if there is no $m\in \N$ such that $(H^L)_m$ is properly colorable with finitely many colors. Another related and already examined question concerning polychromatic edges is the following: what is the smallest $r$ such that every $r$-coloring of $\{1,2,\ldots,n\}$ contains a polychromatic arithmetic progression of size $k$? This is usually referred to as the {\it anti-van der Waerden number} aw$([n],k)$, which is also studied in the case of Abelian groups instead of $[n]$, see e.g. \cite{aw1, aw2}.

We have considered the polychromatic colorability of these, and also of somewhat more general hypergraphs. An additional motivation for this is a connection between the family $\mathcal{B}$ of geometric hypergraphs induced by bottomless rectangles and a special case of the arithmetic progression hypergraphs (Theorem \ref{veges_sorozat}. (1), $M=\{0\}$), realized by Keszegh and Pálvölgyi \cite{temavezetes}. In this paper some more connections between geometric hypergraph families and hypergraph families induced by arithmetic progressions are shown, leading to bounds for the parameter $m_k$ of these hypergraphs.

We denote by $\mathbb{N}$ the set of all natural numbers (including $0$), and $\mathbb{N}^+=\mathbb{N}\setminus\{0\}$.

\begin{definition} 
Let $D,M \subseteq \mathbb{N}$ and denote by $\mathcal{A}_{D,M}$ the family of hypergraphs $A$ such that $S:=V(A) \subseteq \mathbb{N}$ is a finite set, and 
$E(A) \subseteq \{\{a_n\}_{n=0}^{k} \cap S:\{a_n\}_{n=0}^{k} \hspace{0.2 cm}\text{is a finite arithmetic} \newline \text{progression with difference} \hspace{0.2 cm}d \in D\hspace{0.2 cm} \text{such that}\hspace{0.2 cm}\exists m\in \N: a_0-md\in M\}.$

Denote by $\mathcal{A}_{D,M}^{\infty}$ the family of hypergraphs $A$ such that $S:=V(A) \subseteq \mathbb{N}$ is a finite set and $E(A) \subseteq \{\{a_n\}_{n=0}^{\infty} \cap S:\{a_n\}_{n=0}^{\infty} \hspace{0.2 cm}\text{is an infinite arithmetic progression with difference} \hspace{0.2 cm}d \in D\hspace{0.2 cm} \text{such that}\hspace{0.2 cm}\exists m\in \N: a_0-md\in M \}.$
\end{definition}

For the sake of simplicity we use the notations
$m_{D,M}(k):=m_{\mathcal{A}_{D,M}}(k)$ and
$m_{D,M}^\infty(k):=m_{\mathcal{A}_{D,M}^{\infty}}(k)$.

If $D$ can be written in the form $\{\prod_{i=1}^n t_i^{j_i}: j_1,\ldots,j_n\in \mathbb{N}\}$ for some $t_1,\ldots,t_n$ integers, we can determine in most cases whether $m_{D,M}(k)$ and $m_{D,M}^\infty(k)$ are finite or not. We do this with the help of the already defined family $\mathcal{B}$ and the following hypergraph families:

\begin{definition} Denote by $\RR$ the hypergraph family which consists of hypergraphs $H=(V,E)$ such that $V\subset \R^2$ is finite and the edges are induced by axis-parallel rectangles: every $e\in E$ edge can be written in the form $e=V\cap\big(\{(x,y)\in\R^2: y_0\le y \le y_1, x_0\le x\le x_1, \}\big)$ for some $x_0,x_1,y_0,y_1\in\R$. 
\end{definition}

\begin{definition} 
Denote by $\mathcal{T}_3$ the hypergraph family which consists of hypergraphs $H =(V,E)$ such that $V \subset \mathbb{R}^3$ is finite and the edges are induced by translates of octants: every $e\in E$ edge can be written in the form
\[
e = V \cap \{(x,y,z):x_0 \le x,y\le y_0, z\le z_0\}
\]
some $x_0,y_0,z_0 \in \mathbb{R}$. The point $(x_0,y_0,z_0)$ is called the corner of the octant.
\end{definition}

\begin{definition} 
Denote by $\mathcal{T}_4$ the hypergraph family which consists of hypergraphs $H =(V,E)$ such that $V \subset \mathbb{R}^4$ is finite, and the edges are induced by translates of hextants: every $e\in E$ edge can be written in the form
\[
e = V \cap \{(x,y,z,w):x_0 \le x,y\le y_0, z\le z_0, w\le w_0\}
\]
some $x_0,y_0,z_0,w_0 \in \mathbb{R}$.
\end{definition}

Keszegh and Pálvölgyi showed that $m_{\mathcal{T}_3}(k)\le O(k^{5,09})$ \cite{ternyolcad}, and it is known that $m_\RR(k)=\infty$ (in fact, there is not even an integer $m$ such that the elements of $\RR_{=m}$ can be properly $k$-colored \cite[Theorem 3.]{teglalap})
Using this and a remark of Cardinal, Keszegh and Pálvölgyi proved that there is no $m$ such that the elements of $(\mathcal{T}_4)_{=m}$ can all be $k$-colored \cite[Theorem 12.]{hextant}, which yields $m_{\mathcal{T}_4}(k)=\infty$.

These hypergraph families are connected to the hypergraph families defined by arithmetic progressions in the following way: for suitable $C,D$ the vertices of $H\in\A_{C,D}$ or $H\in\A_{C,D}^\infty$ can be put into correspondence with the points of the 2-, 3-, or 4 dimensional space in such a way that the resulting set of points contains all hyperedges as an edge of a geometric hypergraph on this set, or conversely, we define a vertex set $S\subset \mathbb{N}$ while preserving all the edges.
This gives the possibility to improve the colorings of geometric hypergraphs by using hypergraphs induced by arithmetic progressions, and vice versa.

Our results regarding hypergraphs induced by infinite arithmetic progressions are as follows:

\begin{theorem}\label{vegtelen_sorozat}

\begin{enumerate}
    \item 
    Let $D := \{t^i:i \in \mathbb{N}\}$ for some $t \in \mathbb{N} \setminus \{0,1\}$, then
    \[
    m_{D, \mathbb{N}}^{\infty}(k) \leq m_{\mathcal{T}_3}(k).
    \]
    More generally, this also holds if $D = \{d_i: i\in\N, d_i \in \mathbb{N}^+, d_{i-1} | d_{i}\}$, $d_0:=1$.
    \item 
    Let $p,q$ be positive coprime integers, and suppose $M$ contains at most one element of every residue class modulo $pq$, then
    \[
    m_{\{p^iq^j:i,j \in \mathbb{N}\}, M}^{\infty}(k) \leq \max(p,q)\cdot m_{\mathcal{T}_3}.
    \]
    Moreover, if $|M|=1$, then $m_{\{p^iq^j:i,j \in \mathbb{N}\}, M}^{\infty}(k) = m_{\mathcal{T}_3}$.
    \item 
    Let $p_1, p_2, p_3$ be positive coprime integers, then 
    \[
    m_{\left\{ p_1^{j_1}p_2^{j_2}p_3^{j_3}:j_1,j_2,j_3 \in \mathbb{N}\right\},\{0\}}^{\infty}(k) = \infty.
    \]

\end{enumerate}
\end{theorem}

For the case of finite arithmetic progression we have the following:

\begin{theorem}\label{veges_sorozat}
\begin{enumerate}
    \item 
    Suppose that $M$ contains at most one element of every residue class modulo $t$, then
    \[
    m_{\{t^i:i \in \mathbb{N}\}, M}(k) \leq |M|(m_{\mathcal{B}}(k) -1) + 1.
    \]
    \item 
    Let $p,q$ be positive coprime integers, then
    \[
    m_{\{p^iq^j:i,j \in \mathbb{N}\},\{0\}}(k) = \infty.
    \]
\textbf{\textbf{}}\end{enumerate}
\end{theorem}

The proof of case $M=\{0\}$ in $(1)$ is by Keszegh and Pálvölgyi \cite{temavezetes}.

\section{Proofs}

\noindent{\it Notation.}
Directions in the plane are denoted by north, west, south, and east.

\begin{proof} [Proof of Lemma \ref{prop1}.] 
Let $H\in\HH$ be arbitrary, we need to show that $H_{\ge c(k-1)+1}$ can be colored polychromatically with $k$ colors.

First shrink all edges of $H_{\ge c(k-1)+1}$ to be of size exactly $c(k-1)+1$ in such a way that the resulting graph is in $\HH$, we can do this by our assumptions.
Then choose a $c$-shallow hitting set from our hypergraph, color its points to one fixed color, and leave out these vertices from the graph. The resulting graph is still an element of the hypergraph family, and all of its edges can be shrinked to contain exactly $c(k-2)+1$ vertices, and so that we get a graph from $\HH_{=c(k-2)+1}$. Then take another $c$-shallow hitting set, color it to the second color, and so on. After determining this way the first $k-1$ color classes, we still have at least one remaining vertex in every edge, color the remaining vertices to the $k$th color to get a desired $k$-coloring.
\end{proof}

\begin{proof} [Proof of Theorem~\ref{bottomless}] 

Take $\B_{=m}$ with $m\ge 12$. We will show that this family has an element with no 3-shallow hitting set.
First take $m$ points, $P=\{p_1,\ldots,p_m\}$ along a horizontal line, and above each of them, take the construction on Figure \ref{fig:haromsek}.


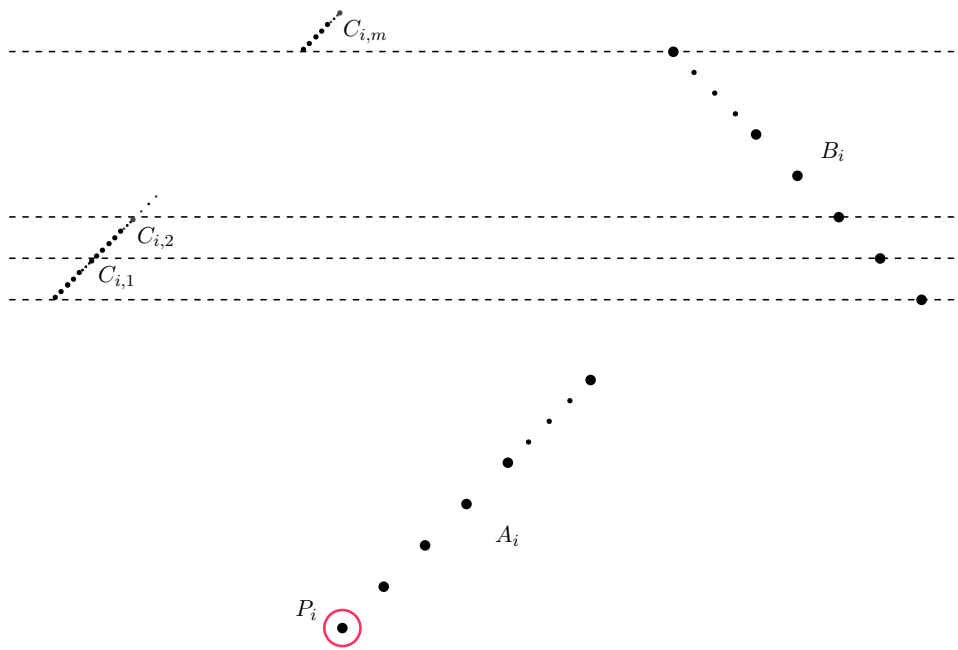
\begin{figure}[h!]
    \centering
    \begin{tikzpicture}[line cap=round,line join=round,>=triangle 45,x=1.1cm,y=1.1cm]
\clip(-4.02,-0.38) rectangle (7.54,7.76);
\draw [line width=0.6pt,dash pattern=on 2pt off 3pt,domain=-4.02:7.54] plot(\x,{(--13.9-0*\x)/3.5});
\draw [line width=0.6pt,dash pattern=on 2pt off 3pt,domain=-4.02:7.54] plot(\x,{(-42.49-0*\x)/-9.5});
\draw [line width=1pt,color=ffttww] (0,0) circle (0.24cm);
\draw (1.75,1.33) node[anchor=north west] [scale=0.8]{$A_i$};
\draw (-3.05,4.47) node[anchor=north west] [scale=0.8]{$C_{i,1}$};
\draw (-2.58,4.95) node[anchor=north west] [scale=0.8]{$C_{i,2}$};
\draw (-0.09,7.45) node[anchor=north west] [scale=0.8]{$C_{i,m}$};
\draw (5.68,5.98) node[anchor=north west] [scale=0.8]{$B_i$ };
\draw (-0.2,0) node[anchor=south east] 
[scale=0.8]{$P_i$ };
\draw [line width=0.6pt,dash pattern=on 2pt off 3pt,domain=-4.02:7.54] plot(\x,{(-47.24-0*\x)/-9.5});
\draw [line width=0.6pt,dash pattern=on 2pt off 3pt,domain=-4.02:7.54] plot(\x,{(-66.24-0*\x)/-9.5});
\begin{scriptsize}
\fill [color=black] (0,0) circle (2pt);
\fill [color=black] (1,1) circle (2pt);
\fill [color=black] (0.5,0.5) circle (2pt);
\fill [color=black] (6.5,4.47) circle (2pt);
\fill [color=black] (-3.47,4) circle (1.0pt);
\fill [color=black] (3,3) circle (2pt);
\fill [color=black] (7,3.97) circle (2pt);
\fill [color=black] (-3.25,4.22) circle (1.0pt);
\fill [color=black] (-3.03,4.44) circle (1.0pt);
\fill [color=black] (2,2) circle (2pt);
\fill [color=black] (1.5,1.5) circle (2pt);
\fill [color=black] (2.5,2.5) circle (1.0pt);
\fill [color=black] (2.75,2.75) circle (1.0pt);
\fill [color=black] (2.25,2.25) circle (1.0pt);
\fill [color=black] (4,6.97) circle (2pt);
\fill [color=black] (4.25,6.72) circle (1.0pt);
\fill [color=black] (4.5,6.47) circle (1.0pt);
\fill [color=black] (4.75,6.22) circle (1.0pt);
\fill [color=black] (5,5.97) circle (2pt);
\fill [color=black] (5.5,5.47) circle (2pt);
\fill [color=black] (6,4.97) circle (2pt);
\fill [color=black] (-3.4,4.07) circle (1.0pt);
\fill [color=black] (-3.32,4.15) circle (1.0pt);
\fill [color=black] (-3.18,4.3) circle (1.0pt);
\fill [color=uuuuuu] (-3.1,4.37) circle (0.5pt);
\fill [color=uuuuuu] (-3.14,4.33) circle (0.5pt);
\fill [color=black] (-3.1,4.37) circle (0.5pt);
\fill [color=black] (-3.06,4.41) circle (0.5pt);
\fill [color=black] (-3.14,4.33) circle (0.5pt);
\fill [color=black] (-2.97,4.5) circle (1.0pt);
\fill [color=black] (-2.9,4.57) circle (1.0pt);
\fill [color=black] (-2.82,4.65) circle (1.0pt);
\fill [color=black] (-2.75,4.72) circle (1.0pt);
\fill [color=black] (-2.68,4.8) circle (1.0pt);
\fill [color=uuuuuu] (-2.53,4.94) circle (1.0pt);
\fill [color=uuuuuu] (-2.6,4.87) circle (0.5pt);
\fill [color=uuuuuu] (-2.64,4.83) circle (0.5pt);
\fill [color=black] (-2.6,4.87) circle (0.5pt);
\fill [color=black] (-2.56,4.91) circle (0.5pt);
\fill [color=uuuuuu] (-2.43,5.04) circle (0.5pt);
\fill [color=black] (-2.34,5.13) circle (0.5pt);
\fill [color=uuuuuu] (-2.25,5.22) circle (0.5pt);
\fill [color=black] (-3.32,4.15) circle (1.0pt);
\fill [color=black] (-0.47,7) circle (1.0pt);
\fill [color=black] (-0.4,7.07) circle (1.0pt);
\fill [color=black] (-0.32,7.15) circle (1.0pt);
\fill [color=black] (-0.25,7.22) circle (1.0pt);
\fill [color=uuuuuu] (-0.14,7.33) circle (0.5pt);
\fill [color=black] (-0.18,7.3) circle (1.0pt);
\fill [color=uuuuuu] (-0.06,7.41) circle (0.5pt);
\fill [color=black] (-0.1,7.37) circle (0.5pt);
\fill [color=uuuuuu] (-0.03,7.44) circle (1.0pt);
\end{scriptsize}
\end{tikzpicture}
    \caption{Construction for Theorem~\ref{bottomless}.}
    \label{fig:haromsek}
\end{figure}

Starting from $p_i$ we have $A_i$ on a diagonal line consisting of $m-2$ points (including $p_i$) with increasing $x$ and $y$ coordinates, to northwest from $A_i$ we have $B_i=\{b_i^1,\ldots,b_i^m\}$ on a diagonal line with increasing $y$, but decreasing $x$ coordinates. To the west from $A_i$, we have $m^2$ points in a diagonal again with increasing $x$ and $y$ coordinates, consisting of $m$ smaller groups $C_{i,j}=\{c_{i,j}^1,\ldots,c_{i,j}^m\}$ (for $j=1,\ldots,m$) of size $m$, with $C_{i,j}$ being vertically between the $b_i^j$ and $b_i^{j+1}$. More precisely, let $(x_1,y_1)\in A_i,$ $(x_2,y_2)=b_i^k\in B_i,$ and $(x_3,y_3)\in C_{i,j}.$ Then $x_3<x_1<x_2$, $y_1<y_2$, $y_1<y_3,$ and we have $y_2<y_3$ exactly when $j\ge k.$

Let $H$ be the hypergraph on this set which contains all $m$-sets contained in a bottomless rectangle as hyperedges. Now for the sake of contradiction suppose that this hypergraph has a 3-shallow hitting set.

There is a point in $P$ which is in the hitting set, because these vertices form a hyperedge. Now we separate two cases:

Case 1. $A_i$ contains at least one further point from the hitting set besides $P$. $B_i$ consists of $m$ points, and can be separated by a bottomless rectangle, so it has a point from the hitting set, say $b_i^j$. We can choose an axis-parallel rectangle that contains exactly $b_i^j$ from $B_i$. Also in $C_{i,j}$ there must be a point from the hitting set, so we can choose the rectangle in a way that its top left corner contains exactly this point from $C_{i,j}$, say $c_{i,j}^l$. Now we got a bottomless rectangle with $m$ points in it, $m-2$ from $A_i$ with at least two points from the hitting set, and $b_i^j$ and $c_{i,j}^l$, both latter two from the hitting set. This is 4 in total; thus, the hitting set is not 3-shallow, a contradiction.

Case 2. $A_i$ does not contain any more points from the hitting set (apart from $p_i$). Now we can take a bottomless rectangle which contains the $m-3$ rightmost points of $A_i$ and any 3 neighbouring points from $B_i$. This implies that any 3 neighbouring points in $B_i$ contain a point from the hitting set, which means $B_i$ contains at least 4 points from the hitting set if $m\ge12$, which is a contradiction.
\end{proof}

\begin{proof}[Proof of Theorem~\ref{strip_dual}]

Take $\lceil \frac{3}{4}k\rceil-1$ copies each of the following axis-parallel strips: $\{(x,y):0<x<2\}$, $\{(x,y):1<x<3\}$, $\{(x,y):0<y<2\}$, and $\{(x,y):1<y<3\}$. Notice that any two of these has a part of their intersection which is disjoint from the other two original strips. (If we do not want vertices derived from the same strip, we can perturb them a little.)

Take that $H$ hypergraph from $\A^*$ which has these $4(\lceil \frac{3}{4}k\rceil-1)$ strips as vertices and all possible hyperedges determined by points. In any coloring of the strips, for each of the 4 original strips there exist strictly more than $\frac{k}{4}$ colors which do not appear in the copies of that strip. By the pigeonhole principle there are two original strips and a color excluded from copies of these two strips, hence $H$ has a hyperedge of size $2(\lceil \frac{3}{4}k\rceil-1)$ which is not polychromatic, implying $2\left(\left\lceil \frac{3}{4}k\right\rceil -1\right) + 1 \leq m_{\A^*}(k)$.
\end{proof}

\begin{proof}[Proof of Theorem~\ref{strip}]

We construct a hypergraph in $\A$ with no 2-shallow hitting sets. The main arrangement of some vertices can be seen on Figure \ref{abra1}.
\begin{figure} [h!]
    \centering
    \includegraphics{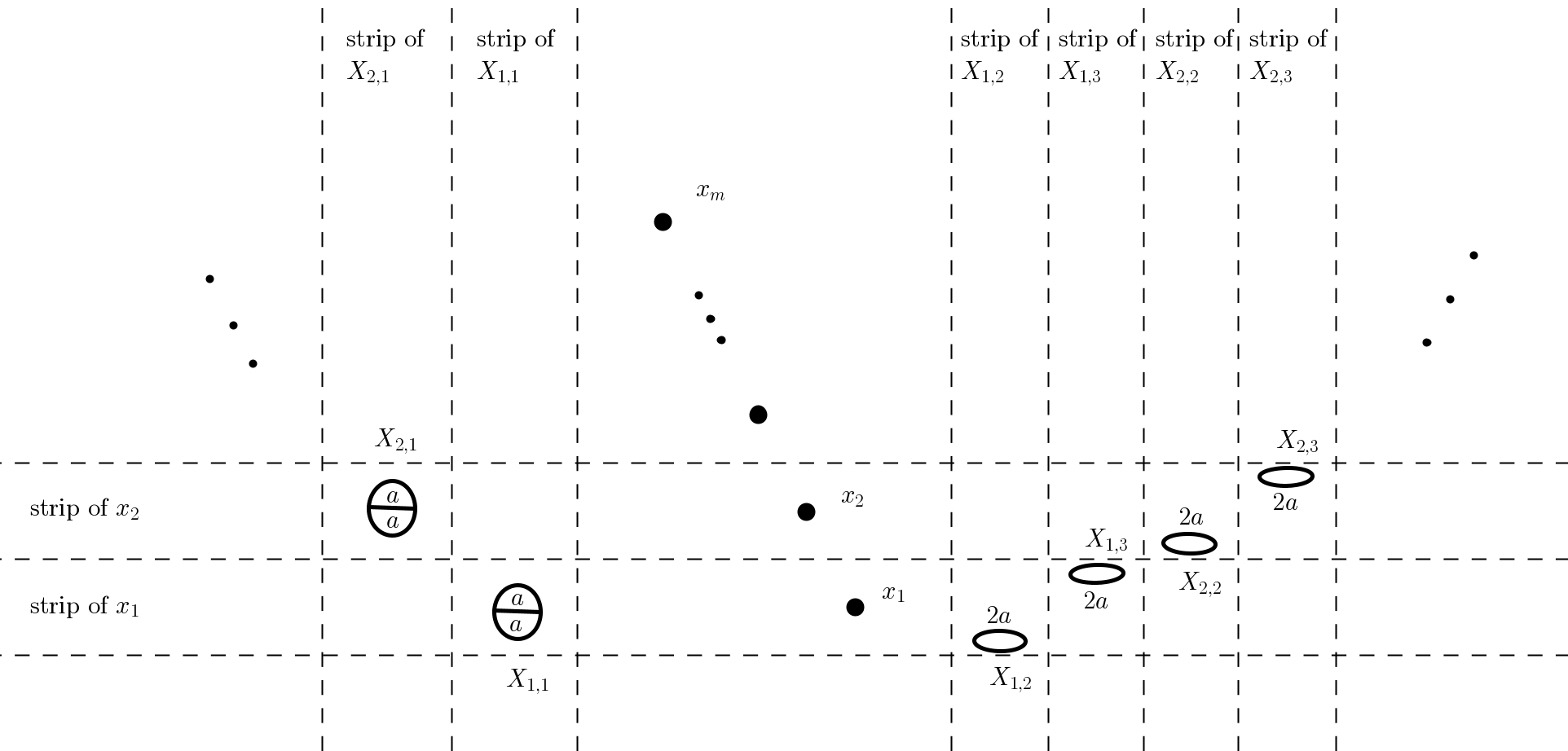}
    \caption{The construction for Theorem~\ref{strip}.}
\label{abra1}
\end{figure}

Let $m$ be large enough, let $a= \lceil \frac{m}{3} \rceil-1$, and $b$ be 8 for $m \equiv 0 \mod 3$, 4 for $m \equiv 1 \mod 3$, and 6 for $m \equiv 2 \mod 3$, this satisfies $m=3a+ \frac{b}{2}-1$. Later we will need that $m\ge 5b$.

First we take $X_0=\{x_1,\ldots,x_m\},$ a set of $m$ points on a southeast to northwest diagonal, as in the middle of Figure \ref{abra1}. We will leave empty a vertical strip containing them. Then we assign disjoint horizontal strips containing each $x_i$, and in all of these horizontal strips, we take 3 sets of points $X_{i,1},X_{i,2},X_{i,3}$ of size $2a$ as shown on Figure \ref{abra1}, reserving disjoint vertical strips containing each of these sets, and with all points from $X_{i,2}$ and $a$ points from $X_{i,1}$ having smaller $y$-coordinates than $x_i$, and all points from $X_{i,3}$ and the other $a$ points from $X_{i,1}$ having larger $y$-coordinates than $x_i$. 

More precisely, the following properties are needed:

Take $3m+1$ pairwise disjoint closed intervals $I_1,I_2,\ldots,I_{3m+1}$ on the $x$-axis, these correspond to the strips of $X_0$ and $\{X_{i,j}\}_{1\le i\le m, 1\le j\le3}$. Take $m$ pairwise disjoint intervals on the $y$-axis, $J_1,\ldots,J_m$, these correspond to the strips of $x_1,\ldots,x_m$. Let $X_0$ be a set of $m$ points such that $\mathrm{conv}(\mathrm{proj}_1(X_0))\subseteq I_{3m+1}$ and $\mathrm{proj}_2(x_i)\in\mathrm{int(J_i)} \, (1\le i\le m).$ Let $\{X_{i,j}\}_{1\le i\le m, 1\le j\le3}$ be sets of $2a$ points each such that $\mathrm{conv}(\mathrm{proj}_1(X_{i,j}))\subseteq I_{(i-1)\cdot3+j}$ and \\
$\text{conv}(\text{proj}_2(X_{i,j}))\subseteq J_i\cap (\text{proj}_2(x_i),\infty)\hfill\text{ if } j=3,$\\
$\text{conv}(\text{proj}_2(X_{i,j}))\subseteq J_i\cap (-\infty,\text{proj}_2(x_i))\hfill\text{ if } j=2,$\\
$|\text{conv}(\text{proj}_2(X_{i,j}))\cap J_i\cap (\text{proj}_2(x_i),\infty)|=|\text{conv}(\text{proj}_2(X_{i,j}))\cap J_i\cap (-\infty,\text{proj}_2(x_i))|=a\text{ if } j=1.$

It is enough to show that we can place additional points in the vertical strips in such a way that it forces the original set of $2a$ points to contain a point from a 2-shallow hitting set. Indeed, if this is the case, there is an element in $X_0$ from the hitting set, say $x_i,$ and $X_{i,1},X_{i,2},$ and $X_{i,3}$ each contain a point from the hitting set. Then there is a horizontal strip with $3a+1\le m$ vertices which contains at least 3 points from the hitting set, namely, $\R\times[\text{proj}_2(x_i),\max(J_i)]$ or $\R\times[\min(J_i),\text{proj}_2(x_i)]$, depending on which half of $X_{i,1}$ contains a point from the hitting set, a contradiction.

For each $X\in\{X_{i,j}: 1\le i\le m, \,1\le j\le3\}$, the construction of this arrangement in a vertical strip is shown on Figure \ref{gigasok}. These points are once again placed in separate horizontal strips for each $X_{i,j}$ (that is, the convex hulls of their projection to the $y$ axis are pairwise disjoint closed intervals, disjoint even with $J_1,\ldots,J_m$). Fix $X\in\{X_{i,j}: 1\le i\le m, \,1\le j\le3\},$ we show that the needed arrangement indeed exists. Take 10 sets of points, $B,C,D,\ldots,K,$ with cardinalities shown on Figure \ref{gigasok}, so that for any elements $(x^B,y^B)\in B, (x^C,y^C)\in C, \ldots, (x^K,y^K)\in K$ and $(x^X,y^X)\in X$ we have 

\begin{center}
    $x^I<x^K<x^J<x^G<x^E<x^D<x^H<x^F<x^C<x^X<x^B,$

    $y^X<y^K<y^J<y^I<y^G<y^H<y^E<y^B<y^D<y^C,$

    $y^H<y^F<y^B.$
\end{center}

\noindent Therefore, if we take neighboring strips on Figure \ref{gigasok}, then the union of the contained sets is captured by an axis-parallel strip.  

Moreover, we arrange the points of $B=\{b_1,b_2,\ldots,b_{m-2a}\}$ and $C=\{c_1,c_2,\ldots,c_{m-2a}\}$ on a diagonal line such that for $b_k=(x_k^B,y_k^B)$ and $c_l=(x_l^C,y_l^C)$ $(1\le k,l \le m-2a)$ we have

\begin{center}
    $x_{m-2a}^C<x_{m-2a-1}^C<\ldots<x_1^C<x_{m-2a}^B<x_{m-2a-1}^B<\ldots<x_1^B,$

    $y_1^B<y_2^B<\ldots<y_{m-2a}^B<y_1^C<y_2^C<\ldots<y_{m-2a}^C.$
\end{center}

\noindent This implies that if we take $B'=\{b_k, b_{k+1},\ldots,b_{m-2a}\}$, the last few elements of $B$, and $C'=\{c_1,c_2,\ldots,c_l\}$, the first few elements of $C,$ then both $B'\cup C'\cup D$ and $B'\cup C'\cup X$ are captured by an axis-parallel strip.


\begin{figure} [h!]
    \centering
    \includegraphics{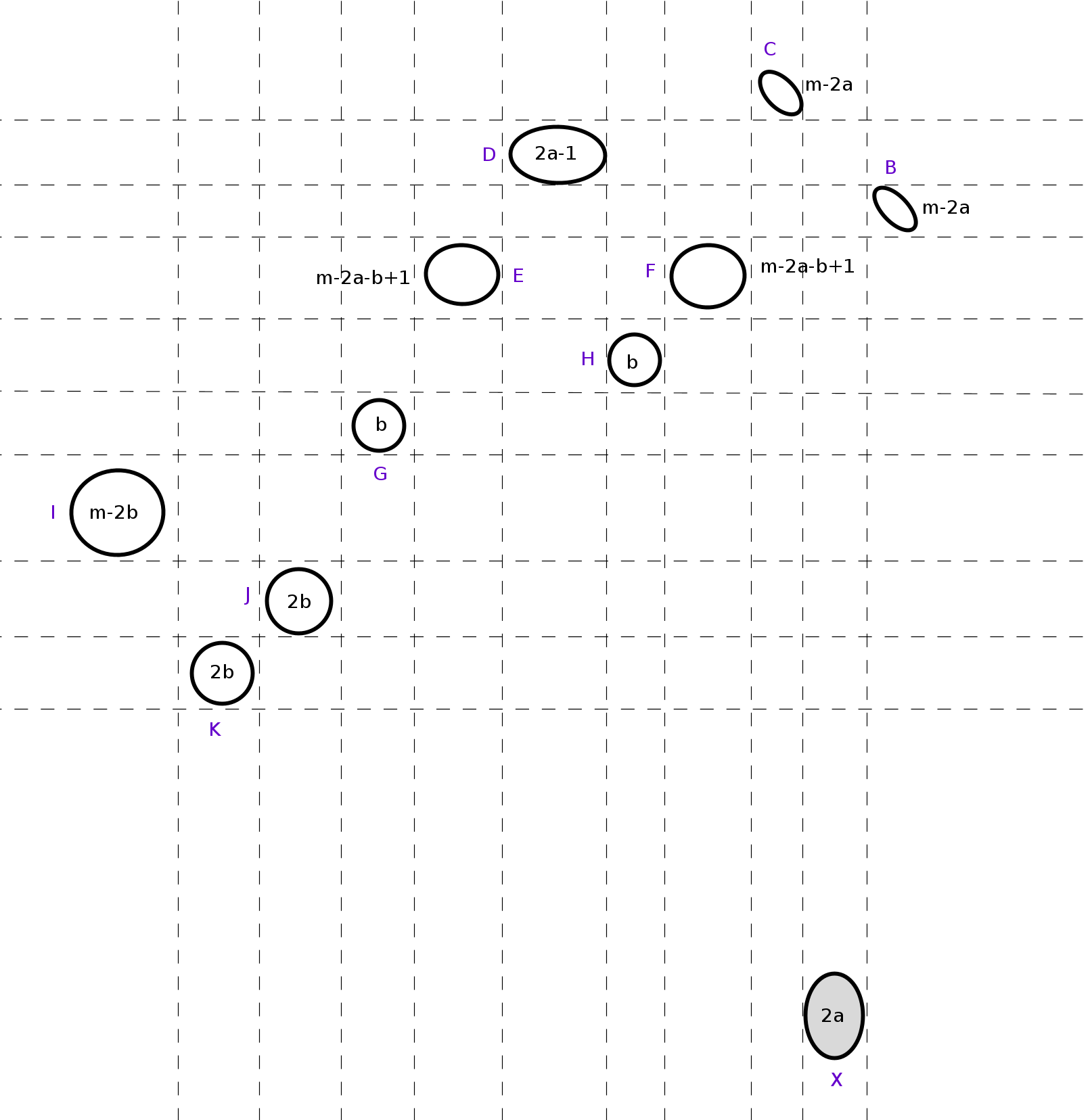}
    \caption{The vertical strip of a $2a$-set $X$}
    \label{gigasok}
\end{figure}

Now let $S$ be a 2-shallow hitting set, and suppose indirectly that $X$ does not contain an element from $S$. Then $B$ contains a point from the hitting set, because $B\cup X$ is captured by a vertical strip of size $m$. Take the largest $k$ such that $b_{k}\in S$, and take $B'=\{b_{k+1},\ldots,b_{m-2a}\}$, and $C'=\{c_1,c_2,\ldots,c_k\}$. Then $B'\cup C'\cup X$ is captured by a vertical strip, $|B'\cup C'\cup X|=m$, and $S\cap(B'\cup X)=\emptyset,$ so $S\cap C'\ne \emptyset$. Therefore, $|S\cap(\{b_k\}\cup B'\cup C')|\ge 2,$ which implies that $D\cap S=\emptyset$, because $\{b_k\}\cup B'\cup C'\cup D$ is captured by a horizontal strip.

Observe that $|D\cup E\cup G|=|D\cup E\cup H|=|D\cup F\cup H|=m$, and that these sets are all captured by an axis-parallel strip, thus $E \cup G$, $E \cup H$ and $F \cup H$ each contain a point from $S$. Notice that $a-\frac{b}{2}=m-2a-b+1=|E|=|F|$, so $|E\cup F\cup H|=2a$, which implies $\big|(E\cup F\cup H)\cap S\big|\le 1$, because otherwise $E \cup F \cup H \cup B$ would form a hyperedge of cardinality $m$ containing at least 3 points from the hitting set.


 This implies $E\cap S=\emptyset$, hence $G$ and $H$ each contain a point from S, so $I\cap S=\est$. Thus $J$ and $K$ each contain a point from $S$, so $|S\cap(G\cup J\cup K)|\ge3.$ Assuming $m\ge 5b$, $G\cup J\cup K$ can be extended to a hyperedge of cardinality $m$, which is a contradiction.
\end{proof}

\begin{proof}[Proof of Theorem~\ref{kereszt}]

For proving the upper bound, color an arbitrary point set in the plane in such a way that all axis-parallel strips of size at least $2k-1$ are polychromatic, this is possible by Theorem 1. in \cite{apstrip}. Then any edge of size at least $4k-3$ contains a horizontal or vertical axis-parallel strip of size at least $2k-1$ and thus polychromatic, implying that the coloring is polychromatic.

For the construction to prove the lower bound, take 8 sets $H_1, H_2, \ldots, H_8$, each consisting of $\lceil \frac{3}{4}k\rceil-1$ points in such an arrangement so that for $(x_i,y_i)\in H_i$ (where $1\le i\le8$) we have 

\begin{center}
    $x_1<x_2<x_3<x_4<x_5<x_6<x_7<x_8$

$y_7<y_5<y_8<y_6<y_3<y_1<y_4<y_2,$  
\end{center}

\noindent see Figure \ref{keresztabra}.

\begin{figure}[h!]
\centering
\begin{tikzpicture}[line cap=round,line join=round,>=triangle 45,x=0.8cm,y=0.8cm]
\clip(-0.5,-3.5) rectangle (9.5,4.5);
\draw(3,1) circle (0.24cm);
\fill [color=black] (3.1,1.1) circle (0.5pt);
\fill [color=black] (3,0.9) circle (0.5pt);
\fill [color=black] (2.9,1) circle (0.5pt);
\fill [color=black] (2.95,1.05) circle (0.5pt);
\draw(2,4) circle (0.24cm);
\fill [color=black] (2.1,4.1) circle (0.5pt);
\fill [color=black] (2,3.9) circle (0.5pt);
\fill [color=black] (1.9,4) circle (0.5pt);
\fill [color=black] (1.95,4.05) circle (0.5pt);
\draw(4,3) circle (0.24cm);
\fill [color=black] (4.1,3.1) circle (0.5pt);
\fill [color=black] (4,2.9) circle (0.5pt);
\fill [color=black] (3.9,3) circle (0.5pt);
\fill [color=black] (3.95,3.05) circle (0.5pt);
\draw(5,-2) circle (0.24cm);
\fill [color=black] (5.1,-2.1) circle (0.5pt);
\fill [color=black] (5,-1.9) circle (0.5pt);
\fill [color=black] (4.9,-2) circle (0.5pt);
\fill [color=black] (4.95,-2.05) circle (0.5pt);
\draw(1,2) circle (0.24cm);
\fill [color=black] (1.1,2.1) circle (0.5pt);
\fill [color=black] (1,1.9) circle (0.5pt);
\fill [color=black] (0.9,2) circle (0.5pt);
\fill [color=black] (0.95,2.05) circle (0.5pt);
\draw(6,0) circle (0.24cm);
\fill [color=black] (6.1,0.1) circle (0.5pt);
\fill [color=black] (6,-0.1) circle (0.5pt);
\fill [color=black] (5.9,0) circle (0.5pt);
\fill [color=black] (5.95,0.05) circle (0.5pt);
\draw(7,-3) circle (0.24cm);
\fill [color=black] (7.1,-3.1) circle (0.5pt);
\fill [color=black] (7,-2.9) circle (0.5pt);
\fill [color=black] (6.9,-3) circle (0.5pt);
\fill [color=black] (6.95,-3.05) circle (0.5pt);
\draw(8,-1) circle (0.24cm);
\fill [color=black] (8.1,-1.1) circle (0.5pt);
\fill [color=black] (8,-0.9) circle (0.5pt);
\fill [color=black] (7.9,-1) circle (0.5pt);
\fill [color=black] (7.95,-1.05) circle (0.5pt);

\draw (1.5,2.5) node [anchor=north west][inner sep=0.75pt]    {$H_1$};

\draw (2.5,4.5) node [anchor=north west][inner sep=0.75pt]    {$H_2$};

\draw (3.5,1.5) node [anchor=north west][inner sep=0.75pt]    {$H_3$};

\draw (4.5,3.5) node [anchor=north west][inner sep=0.75pt]    {$H_4$};

\draw (5.5,-1.5) node [anchor=north west][inner sep=0.75pt]    {$H_5$};

\draw (6.5,0.5) node [anchor=north west][inner sep=0.75pt]    {$H_6$};

\draw (7.5,-2.5) node [anchor=north west][inner sep=0.75pt]    {$H_7$};

\draw (8.5,-0.5) node [anchor=north west][inner sep=0.75pt]    {$H_8$};

\end{tikzpicture}
\caption{Construction for Theorem~\ref{kereszt}.}
\label{keresztabra}
\end{figure}
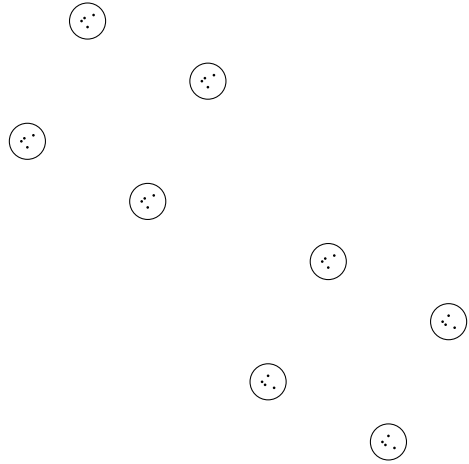

For an arbitrary coloring of the vertices each set out of the eight has strictly more than $\frac{k}{4}$ colors which does not appear in them, so by the pigeonhole principle we have 3 sets with a common missing color. Observe that any three sets can be separated from the others with the union of a horizontal and a vertical strip, which implies $m(k)>3(\lceil \frac{3}{4}k\rceil-1)$.
\end{proof}

\begin{proof}[Proof of Theorem~\ref{s_strip}]

For given $k\ge 2$ take a hypergraph from $\A$ which has a non-poly\-chromatic edge of size at least $m_\A(k)-1$ for any coloring. Take $k(s-1)+1$ copies of the point set that forms the vertices of this hypergraph, and place them along a diagonal line so that all edges of the original hypergraph would be still an edge. For any coloring of the resulting hypergraph, all copies have an edge of size $m_\A(k)-1$ that is not polychromatic, so there is a color which does not appear on the vertices of this hyperedge. By the pigeonhole principle we have $s$ of these hyperedges with the same missing color, these points are contained in the union of $s$ axis-parallel strips, so $s \cdot m_\A(k) -s + 1 \le m_{\A_s}(k)$.

For the other inequality we have to show the existence of a coloring such that any edge of size at least $s \cdot m_\A(k) -s + 1$ contains all $k$ colors. The same coloring as for $m_\A(k)$ satisfies this, because any set of $s \cdot m_\A(k) -s + 1$ points which is contained in the union of $s$ strips has a subset of $m_\A(k)$ points contained in a single strip.
\end{proof}

\begin{proof}[Proof of Theorem~\ref{vegtelen_sorozat}. (1):]
    We will show that for every $A \in \mathcal{A}_{D,\N}^{\infty}$ there exists $T_A \in \mathcal{T}_3$ and $\phi :V(A)\to V(T_A)$ injection such that for every hyperedge $E \subseteq V(A)$ the set $\phi [E]\subseteq V(T_A)$ is also a hyperedge of $T_A$. Therefore, if we can $k$-color $\left(T_A\right)_m$ properly for $m:=m_{\mathcal{T}_3}(k)$, we obtain a polychromatic $k$-coloring of $A_{m}$.
    
    The vertex set of $A \in \mathcal{A}_{D,\N}^{\infty}$ contains finitely many natural numbers. We will assign points in $\mathbb{R}^3$ to them as the vertices of $T_A$ such that for every hyperedge $E \subseteq V(A)$ the set $\phi [E]\subseteq V(T_A)$ can be defined as the intersection of an octant and $V(T_A)$. Fix the difference set D. For the sake of simplicity, we will assign a point in $\mathbb{R}^3$ to every natural number, such that every subset of $\mathbb{N}$ which could be a hyperedge in a hypergraph $A \in \mathcal{A}_{D,\N}^{\infty}$ can be defined by an octant in $\mathbb{R}^3$. 
    
    Firstly, let us discuss the case when $D = \{t^i: i \in \mathbb{N}\}$ for some $t \in \mathbb{N}^+$. We will place axis-parallel squares recursively into each other on the $y$-$z$ plane, each of them corresponding to a natural number. This natural number is placed in the bottom left corner of the square, see Figure \ref{negyzetek}.
    
    \begin{figure} [h!]
    \centering
    \includegraphics[width=1.0\textwidth]{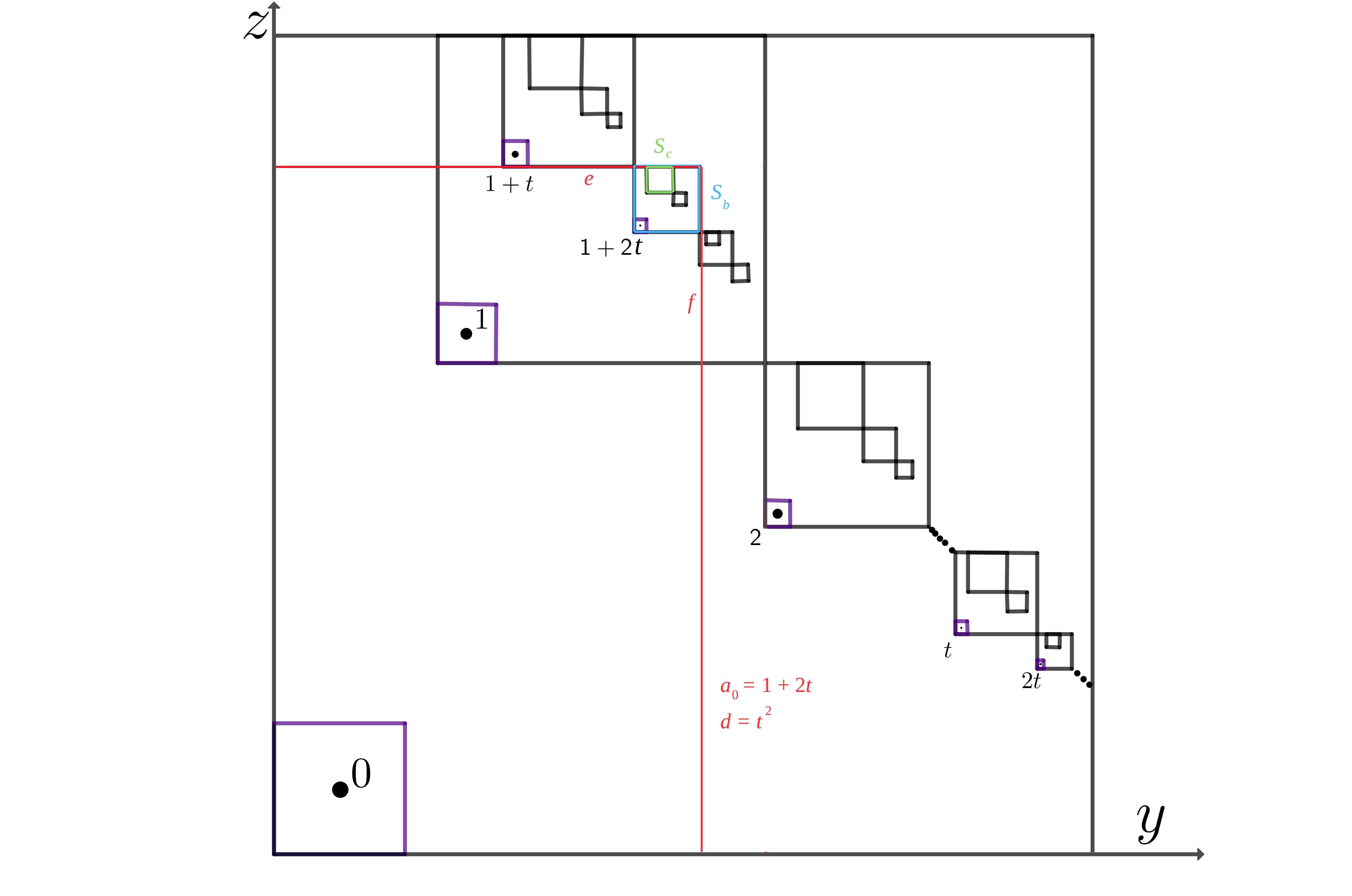}
    \caption{Placement of the squares perpendicular to the $x$-axis}
    \label{negyzetek}
\end{figure}
    
    The $0$th generation of the recursion is a square $S_0$ in the $y$-$z$ plane, corresponding to $0$, say the unit square $[0,1]\times[0,1]$. Place a point $P_0$ in the southwest corner of it. For the first generation, take those natural numbers whose $t$-base form contains only one non-zero bit. Accordingly, the first generation consists of squares $S_1, S_2,...,S_{t-1}, S_{10}, S_{20}, ..., S_{[t-1]0}, S_{100}, S_{200}, ...$ (with the indices being in base $t$) corresponding to $1, 2, ..., t-1, t, 2t, ..., (t-1)t, t^2, 2t^2, ...$, respectively. We place them diagonally (from northwest to southeast) into $S_0$, according to Figure \ref{negyzetek}, and we place a point $P_a$ into the southwest corner of every $S_a$. In terms of coordinates, the $j^\text{th}$ square of the first generation is $\left[1-\frac{1}{j+1}, 1-\frac{1}{j+2}\right]\times\left[\frac{1}{2}+\frac{1}{j+2},\frac{1}{2}+\frac{1}{j+1}\right]$.
    
    The second generation of squares correspond to natural numbers whose $t$-base form contains two nonzero bits. The value and place of the smallest nonzero bit defines that first-generation square in which we put our second-generation square: in $S_a$ where the value of $a$ is $k_1 \cdot t^{i_1}$, we place the squares corresponding to the values $k_1 \cdot t^{i_1} + k_2 \cdot t^{i_2}$ (where $i_1<i_2$), in a diagonal, ordered by the absolute value of $k_2\cdot t^{i_2}$. Take the similarity transformation that maps $S_0$ onto $S_a$ and preserves all directions. The squares of the second generation are placed as the images of the first-generation squares, in the prescribed order. And so on, the $k$-th generation consists of squares corresponding to numbers whose $t$-base form contains exactly $k$ nonzero bits. Note that a square $S_b$ contains another square $S_c$ exactly when the $t$-base form of $c$ is a finat segment of the $t$-base form of $b$. 
    
    As in the 0th and first generation, into every square $S_a$ placed in the $k$-th generation we put a point $P_a$ to the southwest corner. Now we have points in the $y$-$z$ plane assigned to every $n\in\mathbb{N}$, if $n$ has $k$ pieces of nonzero bits in its $t$-base form, we have placed a square for it in the $k$-th generation. Finally, let $\phi (n)=(n,P_n)\in\R^3.$
    
    We have to prove that for every set of form $E = \{a_0+ i \cdot t^j: i \in \mathbb{N}\}$, $a_0, j \in \mathbb{N}$ there exists an octant which contains exactly these points. Firstly, fix $a_0$ and $j$ so that $a_0<t^j$. If we project back our points to the $y$-$z$ plane, we can almost "cut out" the sequence with an axis-parallel quarter-plane; see Figure \ref{negyzetek}. Take the square $S_b$ corresponding to $a_0$. According to the construction, the square $S_c$ corresponding to $a_0 + t^j$ is contained in $S_b$. Let $e$ be the east bordering line of $S_b$ and $f$ be the north bordering line of $S_c$. The lines $e$ and $f$ define four quarter-planes, and let $Q$ be the southwest one with respect to the orientation given by $y$ and $z$. $Q$ contains points corresponding to numbers that are at least $a_0$ and are divisible by $t^j$, or that are smaller than $a_0$. Notice that every number in form of $n=a_0+k \cdot t^j$, $k\in \mathbb{N}$ is contained in $Q$ since only $S_b$ contains numbers congruent to $a_0$ modulo $t^j$, and points above $f$ in $S_b$ correspond to numbers which are the sum of $a_0$ and $r_j$ where $r_j$ has a $t$-base form with $0$ in its $j^{th}$ bit. 
    
    For the case $a_0 \geq t^j$, we can define some $a_{-1}, a_{-2},...,a_{-k}$ such that $0\leq a_{-k} < t^j$, and $\{a_i\}_{i \in \{-k, ..., -1\} \cup \mathbb{N}}$ is still an arithmetic progression, and we can build the same construction for this extended arithmetic progression as in the case $a_0<t^j$. Take the axis-parallel lines $e$ and $f$ as in the previous case, and assume that they are defined by the equations $y=y_0$ and $z=z_0$, respectively. Then the octant in $\mathbb{R}^3$ corresponding to $E$ is the set $\{(x,y,z): x \geq a_0, y \leq y_0, z \leq z_0\}$. 
    
    The general case when $D = \{d_i: i\in\N, d_i \in \mathbb{N}^+, d_{i-1} | d_{i}\}$, $d_0:=1$ is very similar. We use almost the same construction, we just replace the $t$-base form of natural numbers with the following: Assume that for every $k<n$ we have defined a form $k = \sum\limits_{j=0}^{i-1} c_j \cdot d_j$. For $n \in \mathbb{N}$, let $i$ be the maximal index for which $d_i \leq n$. Let $c_i$ be the maximal natural number for which $c_i \cdot d_i \leq n$. Notice that $c_i<\frac{d_{i+1}}{d_i}$. If $k := n-c_i \cdot d_i = \sum\limits_{j=0}^{i-1} c_j \cdot d_j$, write $n$ in form $n = \sum\limits_{j=0}^{i} c_j \cdot d_j$. We build up the construction in the same way, we just replace the $t$-base form of numbers with the sequence $...c_2c_1c_0$.
    \end{proof}

    \begin{proof}[Proof of Theorem~\ref{vegtelen_sorozat}. (2)]
        Firstly, we show that
$m_{\{p^iq^j:i,j \in \mathbb{N}\},\{0\}}^{\infty}(k) \le m_{\mathcal{T}_3}(k)$. To prove this, we construct an injection $\phi:\mathbb{N}\to \mathbb{R}^3$ such that for every fixed $i,j$ the set $\phi[\{a_0+k \cdot p^iq^j: k \in \mathbb{N}, \text{ } p^iq^j\mid a_0\}]$ can be defined by an octant. If $n = k \cdot p^i q^j$ where $p,q$ are not divisors of $k$, let $g_1(n) := i, g_2(n):=j$. Since $p$ and $q$ are relative primes, the form above is unique for every $n$, thus $g_1(n)$ and $g_2(n)$ are well-defined. For $n>0$ define $\phi(n) = (n, 1- \frac{1}{g_1(n)+1}, 1- \frac{1}{g_2(n)+1})$, and let $\phi(0)=(0,1,1)$. The set $\phi[\{a_0+k \cdot p^iq^j: k \in \mathbb{N}, \text{ } p^iq^j\mid a_0\}]$ is the intersection of {\rm Im}$(\phi)$ and the octant  $\{(x,y,z): x \geq a_0, y \geq 1-\frac{1}{i+1}, z \geq 1-\frac{1}{j+1} \}$. Notice that here we use the fact that all considered arithmetic progressions are the final segment of another progression beginning at 0.

Now we turn to the general case. Notice that in the previous construction {\rm Im}$(\phi)$ projected to the $y$-$z$ plane is contained in a square $S$ which has side length $1$. In this case, we build up $pq$ squares on the $y$-$z$ plane, diagonally. The square $S_r$ ($0\leq r < pq$) will contain numbers congruent to $r$ modulo $pq$.  

\begin{figure}[h!]
\begin{center}
    
\tikzset{every picture/.style={line width=0.75pt}} 

\begin{tikzpicture}[x=0.75pt,y=0.75pt,yscale=-1,xscale=1]

\draw    (100.5,290) -- (100.5,12) ;
\draw [shift={(100.5,10)}, rotate = 90] [color={rgb, 255:red, 0; green, 0; blue, 0 }  ][line width=0.75]    (10.93,-3.29) .. controls (6.95,-1.4) and (3.31,-0.3) .. (0,0) .. controls (3.31,0.3) and (6.95,1.4) .. (10.93,3.29)   ;
\draw    (89.5,70) -- (377.5,70) ;
\draw [shift={(379.5,70)}, rotate = 180] [color={rgb, 255:red, 0; green, 0; blue, 0 }  ][line width=0.75]    (10.93,-3.29) .. controls (6.95,-1.4) and (3.31,-0.3) .. (0,0) .. controls (3.31,0.3) and (6.95,1.4) .. (10.93,3.29)   ;
\draw   (180.5,110) -- (220.5,110) -- (220.5,150) -- (180.5,150) -- cycle ;
\draw   (140.5,70) -- (180.5,70) -- (180.5,110) -- (140.5,110) -- cycle ;
\draw   (100.5,30) -- (140.5,30) -- (140.5,70) -- (100.5,70) -- cycle ;
\draw   (300.5,230) -- (340.5,230) -- (340.5,270) -- (300.5,270) -- cycle ;
\draw  [fill={rgb, 255:red, 0; green, 0; blue, 0 }  ,fill opacity=1 ] (120.5,50) .. controls (120.5,49.72) and (120.28,49.5) .. (120,49.5) .. controls (119.72,49.5) and (119.5,49.72) .. (119.5,50) .. controls (119.5,50.28) and (119.72,50.5) .. (120,50.5) .. controls (120.28,50.5) and (120.5,50.28) .. (120.5,50) -- cycle ;
\draw  [fill={rgb, 255:red, 0; green, 0; blue, 0 }  ,fill opacity=1 ] (120.55,39.82) .. controls (120.55,39.54) and (120.32,39.32) .. (120.05,39.32) .. controls (119.77,39.32) and (119.55,39.54) .. (119.55,39.82) .. controls (119.55,40.09) and (119.77,40.32) .. (120.05,40.32) .. controls (120.32,40.32) and (120.55,40.09) .. (120.55,39.82) -- cycle ;
\draw  [fill={rgb, 255:red, 0; green, 0; blue, 0 }  ,fill opacity=1 ] (120.55,43.64) .. controls (120.55,43.36) and (120.32,43.14) .. (120.05,43.14) .. controls (119.77,43.14) and (119.55,43.36) .. (119.55,43.64) .. controls (119.55,43.91) and (119.77,44.14) .. (120.05,44.14) .. controls (120.32,44.14) and (120.55,43.91) .. (120.55,43.64) -- cycle ;
\draw  [fill={rgb, 255:red, 0; green, 0; blue, 0 }  ,fill opacity=1 ] (120.55,36.91) .. controls (120.55,36.63) and (120.32,36.41) .. (120.05,36.41) .. controls (119.77,36.41) and (119.55,36.63) .. (119.55,36.91) .. controls (119.55,37.19) and (119.77,37.41) .. (120.05,37.41) .. controls (120.32,37.41) and (120.55,37.19) .. (120.55,36.91) -- cycle ;
\draw  [fill={rgb, 255:red, 0; green, 0; blue, 0 }  ,fill opacity=1 ] (127.09,50) .. controls (127.09,49.72) and (126.87,49.5) .. (126.59,49.5) .. controls (126.31,49.5) and (126.09,49.72) .. (126.09,50) .. controls (126.09,50.28) and (126.31,50.5) .. (126.59,50.5) .. controls (126.87,50.5) and (127.09,50.28) .. (127.09,50) -- cycle ;
\draw  [fill={rgb, 255:red, 0; green, 0; blue, 0 }  ,fill opacity=1 ] (127.09,43.64) .. controls (127.09,43.36) and (126.87,43.14) .. (126.59,43.14) .. controls (126.31,43.14) and (126.09,43.36) .. (126.09,43.64) .. controls (126.09,43.91) and (126.31,44.14) .. (126.59,44.14) .. controls (126.87,44.14) and (127.09,43.91) .. (127.09,43.64) -- cycle ;
\draw  [fill={rgb, 255:red, 0; green, 0; blue, 0 }  ,fill opacity=1 ] (127.09,40) .. controls (127.09,39.72) and (126.87,39.5) .. (126.59,39.5) .. controls (126.31,39.5) and (126.09,39.72) .. (126.09,40) .. controls (126.09,40.28) and (126.31,40.5) .. (126.59,40.5) .. controls (126.87,40.5) and (127.09,40.28) .. (127.09,40) -- cycle ;
\draw  [fill={rgb, 255:red, 0; green, 0; blue, 0 }  ,fill opacity=1 ] (127.05,36.87) .. controls (127.05,36.59) and (126.83,36.37) .. (126.55,36.37) .. controls (126.27,36.37) and (126.05,36.59) .. (126.05,36.87) .. controls (126.05,37.14) and (126.27,37.37) .. (126.55,37.37) .. controls (126.83,37.37) and (127.05,37.14) .. (127.05,36.87) -- cycle ;
\draw  [fill={rgb, 255:red, 0; green, 0; blue, 0 }  ,fill opacity=1 ] (130.87,50) .. controls (130.87,49.72) and (130.64,49.5) .. (130.37,49.5) .. controls (130.09,49.5) and (129.87,49.72) .. (129.87,50) .. controls (129.87,50.28) and (130.09,50.5) .. (130.37,50.5) .. controls (130.64,50.5) and (130.87,50.28) .. (130.87,50) -- cycle ;
\draw  [fill={rgb, 255:red, 0; green, 0; blue, 0 }  ,fill opacity=1 ] (130.87,43.82) .. controls (130.87,43.54) and (130.64,43.32) .. (130.37,43.32) .. controls (130.09,43.32) and (129.87,43.54) .. (129.87,43.82) .. controls (129.87,44.09) and (130.09,44.32) .. (130.37,44.32) .. controls (130.64,44.32) and (130.87,44.09) .. (130.87,43.82) -- cycle ;
\draw  [fill={rgb, 255:red, 0; green, 0; blue, 0 }  ,fill opacity=1 ] (130.87,40.26) .. controls (130.87,39.99) and (130.64,39.76) .. (130.37,39.76) .. controls (130.09,39.76) and (129.87,39.99) .. (129.87,40.26) .. controls (129.87,40.54) and (130.09,40.76) .. (130.37,40.76) .. controls (130.64,40.76) and (130.87,40.54) .. (130.87,40.26) -- cycle ;
\draw  [fill={rgb, 255:red, 0; green, 0; blue, 0 }  ,fill opacity=1 ] (130.87,36.95) .. controls (130.87,36.67) and (130.64,36.45) .. (130.37,36.45) .. controls (130.09,36.45) and (129.87,36.67) .. (129.87,36.95) .. controls (129.87,37.23) and (130.09,37.45) .. (130.37,37.45) .. controls (130.64,37.45) and (130.87,37.23) .. (130.87,36.95) -- cycle ;
\draw  [fill={rgb, 255:red, 0; green, 0; blue, 0 }  ,fill opacity=1 ] (134,50.04) .. controls (134,49.76) and (133.78,49.54) .. (133.5,49.54) .. controls (133.22,49.54) and (133,49.76) .. (133,50.04) .. controls (133,50.32) and (133.22,50.54) .. (133.5,50.54) .. controls (133.78,50.54) and (134,50.32) .. (134,50.04) -- cycle ;
\draw  [fill={rgb, 255:red, 0; green, 0; blue, 0 }  ,fill opacity=1 ] (134.04,43.86) .. controls (134.04,43.58) and (133.82,43.36) .. (133.54,43.36) .. controls (133.26,43.36) and (133.04,43.58) .. (133.04,43.86) .. controls (133.04,44.13) and (133.26,44.36) .. (133.54,44.36) .. controls (133.82,44.36) and (134.04,44.13) .. (134.04,43.86) -- cycle ;
\draw  [fill={rgb, 255:red, 0; green, 0; blue, 0 }  ,fill opacity=1 ] (134.22,40.08) .. controls (134.22,39.8) and (134,39.58) .. (133.72,39.58) .. controls (133.45,39.58) and (133.22,39.8) .. (133.22,40.08) .. controls (133.22,40.36) and (133.45,40.58) .. (133.72,40.58) .. controls (134,40.58) and (134.22,40.36) .. (134.22,40.08) -- cycle ;
\draw  [fill={rgb, 255:red, 0; green, 0; blue, 0 }  ,fill opacity=1 ] (134.22,36.95) .. controls (134.22,36.67) and (134,36.45) .. (133.72,36.45) .. controls (133.45,36.45) and (133.22,36.67) .. (133.22,36.95) .. controls (133.22,37.23) and (133.45,37.45) .. (133.72,37.45) .. controls (134,37.45) and (134.22,37.23) .. (134.22,36.95) -- cycle ;
\draw  [fill={rgb, 255:red, 0; green, 0; blue, 0 }  ,fill opacity=1 ] (141,30) .. controls (141,29.72) and (140.78,29.5) .. (140.5,29.5) .. controls (140.22,29.5) and (140,29.72) .. (140,30) .. controls (140,30.28) and (140.22,30.5) .. (140.5,30.5) .. controls (140.78,30.5) and (141,30.28) .. (141,30) -- cycle ;
\draw  [fill={rgb, 255:red, 0; green, 0; blue, 0 }  ,fill opacity=1 ] (180.5,70.25) .. controls (180.5,70.11) and (180.39,70) .. (180.25,70) .. controls (180.11,70) and (180,70.11) .. (180,70.25) .. controls (180,70.39) and (180.11,70.5) .. (180.25,70.5) .. controls (180.39,70.5) and (180.5,70.39) .. (180.5,70.25) -- cycle ;
\draw  [fill={rgb, 255:red, 0; green, 0; blue, 0 }  ,fill opacity=1 ] (160.5,90) .. controls (160.5,89.72) and (160.28,89.5) .. (160,89.5) .. controls (159.72,89.5) and (159.5,89.72) .. (159.5,90) .. controls (159.5,90.28) and (159.72,90.5) .. (160,90.5) .. controls (160.28,90.5) and (160.5,90.28) .. (160.5,90) -- cycle ;
\draw  [fill={rgb, 255:red, 0; green, 0; blue, 0 }  ,fill opacity=1 ] (160.55,79.82) .. controls (160.55,79.54) and (160.32,79.32) .. (160.05,79.32) .. controls (159.77,79.32) and (159.55,79.54) .. (159.55,79.82) .. controls (159.55,80.09) and (159.77,80.32) .. (160.05,80.32) .. controls (160.32,80.32) and (160.55,80.09) .. (160.55,79.82) -- cycle ;
\draw  [fill={rgb, 255:red, 0; green, 0; blue, 0 }  ,fill opacity=1 ] (160.55,83.64) .. controls (160.55,83.36) and (160.32,83.14) .. (160.05,83.14) .. controls (159.77,83.14) and (159.55,83.36) .. (159.55,83.64) .. controls (159.55,83.91) and (159.77,84.14) .. (160.05,84.14) .. controls (160.32,84.14) and (160.55,83.91) .. (160.55,83.64) -- cycle ;
\draw  [fill={rgb, 255:red, 0; green, 0; blue, 0 }  ,fill opacity=1 ] (160.55,76.91) .. controls (160.55,76.63) and (160.32,76.41) .. (160.05,76.41) .. controls (159.77,76.41) and (159.55,76.63) .. (159.55,76.91) .. controls (159.55,77.19) and (159.77,77.41) .. (160.05,77.41) .. controls (160.32,77.41) and (160.55,77.19) .. (160.55,76.91) -- cycle ;
\draw  [fill={rgb, 255:red, 0; green, 0; blue, 0 }  ,fill opacity=1 ] (167.09,90) .. controls (167.09,89.72) and (166.87,89.5) .. (166.59,89.5) .. controls (166.31,89.5) and (166.09,89.72) .. (166.09,90) .. controls (166.09,90.28) and (166.31,90.5) .. (166.59,90.5) .. controls (166.87,90.5) and (167.09,90.28) .. (167.09,90) -- cycle ;
\draw  [fill={rgb, 255:red, 0; green, 0; blue, 0 }  ,fill opacity=1 ] (167.09,83.64) .. controls (167.09,83.36) and (166.87,83.14) .. (166.59,83.14) .. controls (166.31,83.14) and (166.09,83.36) .. (166.09,83.64) .. controls (166.09,83.91) and (166.31,84.14) .. (166.59,84.14) .. controls (166.87,84.14) and (167.09,83.91) .. (167.09,83.64) -- cycle ;
\draw  [fill={rgb, 255:red, 0; green, 0; blue, 0 }  ,fill opacity=1 ] (167.09,80) .. controls (167.09,79.72) and (166.87,79.5) .. (166.59,79.5) .. controls (166.31,79.5) and (166.09,79.72) .. (166.09,80) .. controls (166.09,80.28) and (166.31,80.5) .. (166.59,80.5) .. controls (166.87,80.5) and (167.09,80.28) .. (167.09,80) -- cycle ;
\draw  [fill={rgb, 255:red, 0; green, 0; blue, 0 }  ,fill opacity=1 ] (167.05,76.87) .. controls (167.05,76.59) and (166.83,76.37) .. (166.55,76.37) .. controls (166.27,76.37) and (166.05,76.59) .. (166.05,76.87) .. controls (166.05,77.14) and (166.27,77.37) .. (166.55,77.37) .. controls (166.83,77.37) and (167.05,77.14) .. (167.05,76.87) -- cycle ;
\draw  [fill={rgb, 255:red, 0; green, 0; blue, 0 }  ,fill opacity=1 ] (170.87,90) .. controls (170.87,89.72) and (170.64,89.5) .. (170.37,89.5) .. controls (170.09,89.5) and (169.87,89.72) .. (169.87,90) .. controls (169.87,90.28) and (170.09,90.5) .. (170.37,90.5) .. controls (170.64,90.5) and (170.87,90.28) .. (170.87,90) -- cycle ;
\draw  [fill={rgb, 255:red, 0; green, 0; blue, 0 }  ,fill opacity=1 ] (170.87,83.82) .. controls (170.87,83.54) and (170.64,83.32) .. (170.37,83.32) .. controls (170.09,83.32) and (169.87,83.54) .. (169.87,83.82) .. controls (169.87,84.09) and (170.09,84.32) .. (170.37,84.32) .. controls (170.64,84.32) and (170.87,84.09) .. (170.87,83.82) -- cycle ;
\draw  [fill={rgb, 255:red, 0; green, 0; blue, 0 }  ,fill opacity=1 ] (170.87,80.26) .. controls (170.87,79.99) and (170.64,79.76) .. (170.37,79.76) .. controls (170.09,79.76) and (169.87,79.99) .. (169.87,80.26) .. controls (169.87,80.54) and (170.09,80.76) .. (170.37,80.76) .. controls (170.64,80.76) and (170.87,80.54) .. (170.87,80.26) -- cycle ;
\draw  [fill={rgb, 255:red, 0; green, 0; blue, 0 }  ,fill opacity=1 ] (170.87,76.95) .. controls (170.87,76.67) and (170.64,76.45) .. (170.37,76.45) .. controls (170.09,76.45) and (169.87,76.67) .. (169.87,76.95) .. controls (169.87,77.23) and (170.09,77.45) .. (170.37,77.45) .. controls (170.64,77.45) and (170.87,77.23) .. (170.87,76.95) -- cycle ;
\draw  [fill={rgb, 255:red, 0; green, 0; blue, 0 }  ,fill opacity=1 ] (174,90.04) .. controls (174,89.76) and (173.78,89.54) .. (173.5,89.54) .. controls (173.22,89.54) and (173,89.76) .. (173,90.04) .. controls (173,90.32) and (173.22,90.54) .. (173.5,90.54) .. controls (173.78,90.54) and (174,90.32) .. (174,90.04) -- cycle ;
\draw  [fill={rgb, 255:red, 0; green, 0; blue, 0 }  ,fill opacity=1 ] (174.04,83.86) .. controls (174.04,83.58) and (173.82,83.36) .. (173.54,83.36) .. controls (173.26,83.36) and (173.04,83.58) .. (173.04,83.86) .. controls (173.04,84.13) and (173.26,84.36) .. (173.54,84.36) .. controls (173.82,84.36) and (174.04,84.13) .. (174.04,83.86) -- cycle ;
\draw  [fill={rgb, 255:red, 0; green, 0; blue, 0 }  ,fill opacity=1 ] (174.22,80.08) .. controls (174.22,79.8) and (174,79.58) .. (173.72,79.58) .. controls (173.45,79.58) and (173.22,79.8) .. (173.22,80.08) .. controls (173.22,80.36) and (173.45,80.58) .. (173.72,80.58) .. controls (174,80.58) and (174.22,80.36) .. (174.22,80.08) -- cycle ;
\draw  [fill={rgb, 255:red, 0; green, 0; blue, 0 }  ,fill opacity=1 ] (174.22,76.95) .. controls (174.22,76.67) and (174,76.45) .. (173.72,76.45) .. controls (173.45,76.45) and (173.22,76.67) .. (173.22,76.95) .. controls (173.22,77.23) and (173.45,77.45) .. (173.72,77.45) .. controls (174,77.45) and (174.22,77.23) .. (174.22,76.95) -- cycle ;
\draw  [fill={rgb, 255:red, 0; green, 0; blue, 0 }  ,fill opacity=1 ] (200.5,130) .. controls (200.5,129.72) and (200.28,129.5) .. (200,129.5) .. controls (199.72,129.5) and (199.5,129.72) .. (199.5,130) .. controls (199.5,130.28) and (199.72,130.5) .. (200,130.5) .. controls (200.28,130.5) and (200.5,130.28) .. (200.5,130) -- cycle ;
\draw  [fill={rgb, 255:red, 0; green, 0; blue, 0 }  ,fill opacity=1 ] (200.55,119.82) .. controls (200.55,119.54) and (200.32,119.32) .. (200.05,119.32) .. controls (199.77,119.32) and (199.55,119.54) .. (199.55,119.82) .. controls (199.55,120.09) and (199.77,120.32) .. (200.05,120.32) .. controls (200.32,120.32) and (200.55,120.09) .. (200.55,119.82) -- cycle ;
\draw  [fill={rgb, 255:red, 0; green, 0; blue, 0 }  ,fill opacity=1 ] (200.55,123.64) .. controls (200.55,123.36) and (200.32,123.14) .. (200.05,123.14) .. controls (199.77,123.14) and (199.55,123.36) .. (199.55,123.64) .. controls (199.55,123.91) and (199.77,124.14) .. (200.05,124.14) .. controls (200.32,124.14) and (200.55,123.91) .. (200.55,123.64) -- cycle ;
\draw  [fill={rgb, 255:red, 0; green, 0; blue, 0 }  ,fill opacity=1 ] (200.55,116.91) .. controls (200.55,116.63) and (200.32,116.41) .. (200.05,116.41) .. controls (199.77,116.41) and (199.55,116.63) .. (199.55,116.91) .. controls (199.55,117.19) and (199.77,117.41) .. (200.05,117.41) .. controls (200.32,117.41) and (200.55,117.19) .. (200.55,116.91) -- cycle ;
\draw  [fill={rgb, 255:red, 0; green, 0; blue, 0 }  ,fill opacity=1 ] (207.09,130) .. controls (207.09,129.72) and (206.87,129.5) .. (206.59,129.5) .. controls (206.31,129.5) and (206.09,129.72) .. (206.09,130) .. controls (206.09,130.28) and (206.31,130.5) .. (206.59,130.5) .. controls (206.87,130.5) and (207.09,130.28) .. (207.09,130) -- cycle ;
\draw  [fill={rgb, 255:red, 0; green, 0; blue, 0 }  ,fill opacity=1 ] (207.09,123.64) .. controls (207.09,123.36) and (206.87,123.14) .. (206.59,123.14) .. controls (206.31,123.14) and (206.09,123.36) .. (206.09,123.64) .. controls (206.09,123.91) and (206.31,124.14) .. (206.59,124.14) .. controls (206.87,124.14) and (207.09,123.91) .. (207.09,123.64) -- cycle ;
\draw  [fill={rgb, 255:red, 0; green, 0; blue, 0 }  ,fill opacity=1 ] (207.09,120) .. controls (207.09,119.72) and (206.87,119.5) .. (206.59,119.5) .. controls (206.31,119.5) and (206.09,119.72) .. (206.09,120) .. controls (206.09,120.28) and (206.31,120.5) .. (206.59,120.5) .. controls (206.87,120.5) and (207.09,120.28) .. (207.09,120) -- cycle ;
\draw  [fill={rgb, 255:red, 0; green, 0; blue, 0 }  ,fill opacity=1 ] (207.05,116.87) .. controls (207.05,116.59) and (206.83,116.37) .. (206.55,116.37) .. controls (206.27,116.37) and (206.05,116.59) .. (206.05,116.87) .. controls (206.05,117.14) and (206.27,117.37) .. (206.55,117.37) .. controls (206.83,117.37) and (207.05,117.14) .. (207.05,116.87) -- cycle ;
\draw  [fill={rgb, 255:red, 0; green, 0; blue, 0 }  ,fill opacity=1 ] (210.87,130) .. controls (210.87,129.72) and (210.64,129.5) .. (210.37,129.5) .. controls (210.09,129.5) and (209.87,129.72) .. (209.87,130) .. controls (209.87,130.28) and (210.09,130.5) .. (210.37,130.5) .. controls (210.64,130.5) and (210.87,130.28) .. (210.87,130) -- cycle ;
\draw  [fill={rgb, 255:red, 0; green, 0; blue, 0 }  ,fill opacity=1 ] (210.87,123.82) .. controls (210.87,123.54) and (210.64,123.32) .. (210.37,123.32) .. controls (210.09,123.32) and (209.87,123.54) .. (209.87,123.82) .. controls (209.87,124.09) and (210.09,124.32) .. (210.37,124.32) .. controls (210.64,124.32) and (210.87,124.09) .. (210.87,123.82) -- cycle ;
\draw  [fill={rgb, 255:red, 0; green, 0; blue, 0 }  ,fill opacity=1 ] (210.87,120.26) .. controls (210.87,119.99) and (210.64,119.76) .. (210.37,119.76) .. controls (210.09,119.76) and (209.87,119.99) .. (209.87,120.26) .. controls (209.87,120.54) and (210.09,120.76) .. (210.37,120.76) .. controls (210.64,120.76) and (210.87,120.54) .. (210.87,120.26) -- cycle ;
\draw  [fill={rgb, 255:red, 0; green, 0; blue, 0 }  ,fill opacity=1 ] (210.87,116.95) .. controls (210.87,116.67) and (210.64,116.45) .. (210.37,116.45) .. controls (210.09,116.45) and (209.87,116.67) .. (209.87,116.95) .. controls (209.87,117.23) and (210.09,117.45) .. (210.37,117.45) .. controls (210.64,117.45) and (210.87,117.23) .. (210.87,116.95) -- cycle ;
\draw  [fill={rgb, 255:red, 0; green, 0; blue, 0 }  ,fill opacity=1 ] (214,130.04) .. controls (214,129.76) and (213.78,129.54) .. (213.5,129.54) .. controls (213.22,129.54) and (213,129.76) .. (213,130.04) .. controls (213,130.32) and (213.22,130.54) .. (213.5,130.54) .. controls (213.78,130.54) and (214,130.32) .. (214,130.04) -- cycle ;
\draw  [fill={rgb, 255:red, 0; green, 0; blue, 0 }  ,fill opacity=1 ] (214.04,123.86) .. controls (214.04,123.58) and (213.82,123.36) .. (213.54,123.36) .. controls (213.26,123.36) and (213.04,123.58) .. (213.04,123.86) .. controls (213.04,124.13) and (213.26,124.36) .. (213.54,124.36) .. controls (213.82,124.36) and (214.04,124.13) .. (214.04,123.86) -- cycle ;
\draw  [fill={rgb, 255:red, 0; green, 0; blue, 0 }  ,fill opacity=1 ] (214.22,120.08) .. controls (214.22,119.8) and (214,119.58) .. (213.72,119.58) .. controls (213.45,119.58) and (213.22,119.8) .. (213.22,120.08) .. controls (213.22,120.36) and (213.45,120.58) .. (213.72,120.58) .. controls (214,120.58) and (214.22,120.36) .. (214.22,120.08) -- cycle ;
\draw  [fill={rgb, 255:red, 0; green, 0; blue, 0 }  ,fill opacity=1 ] (214.22,116.95) .. controls (214.22,116.67) and (214,116.45) .. (213.72,116.45) .. controls (213.45,116.45) and (213.22,116.67) .. (213.22,116.95) .. controls (213.22,117.23) and (213.45,117.45) .. (213.72,117.45) .. controls (214,117.45) and (214.22,117.23) .. (214.22,116.95) -- cycle ;
\draw  [fill={rgb, 255:red, 0; green, 0; blue, 0 }  ,fill opacity=1 ] (320.56,249.96) .. controls (320.56,249.69) and (320.33,249.46) .. (320.06,249.46) .. controls (319.78,249.46) and (319.56,249.69) .. (319.56,249.96) .. controls (319.56,250.24) and (319.78,250.46) .. (320.06,250.46) .. controls (320.33,250.46) and (320.56,250.24) .. (320.56,249.96) -- cycle ;
\draw  [fill={rgb, 255:red, 0; green, 0; blue, 0 }  ,fill opacity=1 ] (320.6,239.78) .. controls (320.6,239.51) and (320.38,239.28) .. (320.1,239.28) .. controls (319.82,239.28) and (319.6,239.51) .. (319.6,239.78) .. controls (319.6,240.06) and (319.82,240.28) .. (320.1,240.28) .. controls (320.38,240.28) and (320.6,240.06) .. (320.6,239.78) -- cycle ;
\draw  [fill={rgb, 255:red, 0; green, 0; blue, 0 }  ,fill opacity=1 ] (320.6,243.6) .. controls (320.6,243.32) and (320.38,243.1) .. (320.1,243.1) .. controls (319.82,243.1) and (319.6,243.32) .. (319.6,243.6) .. controls (319.6,243.88) and (319.82,244.1) .. (320.1,244.1) .. controls (320.38,244.1) and (320.6,243.88) .. (320.6,243.6) -- cycle ;
\draw  [fill={rgb, 255:red, 0; green, 0; blue, 0 }  ,fill opacity=1 ] (320.6,236.87) .. controls (320.6,236.6) and (320.38,236.37) .. (320.1,236.37) .. controls (319.82,236.37) and (319.6,236.6) .. (319.6,236.87) .. controls (319.6,237.15) and (319.82,237.37) .. (320.1,237.37) .. controls (320.38,237.37) and (320.6,237.15) .. (320.6,236.87) -- cycle ;
\draw  [fill={rgb, 255:red, 0; green, 0; blue, 0 }  ,fill opacity=1 ] (327.15,249.96) .. controls (327.15,249.69) and (326.92,249.46) .. (326.65,249.46) .. controls (326.37,249.46) and (326.15,249.69) .. (326.15,249.96) .. controls (326.15,250.24) and (326.37,250.46) .. (326.65,250.46) .. controls (326.92,250.46) and (327.15,250.24) .. (327.15,249.96) -- cycle ;
\draw  [fill={rgb, 255:red, 0; green, 0; blue, 0 }  ,fill opacity=1 ] (327.15,243.6) .. controls (327.15,243.32) and (326.92,243.1) .. (326.65,243.1) .. controls (326.37,243.1) and (326.15,243.32) .. (326.15,243.6) .. controls (326.15,243.88) and (326.37,244.1) .. (326.65,244.1) .. controls (326.92,244.1) and (327.15,243.88) .. (327.15,243.6) -- cycle ;
\draw  [fill={rgb, 255:red, 0; green, 0; blue, 0 }  ,fill opacity=1 ] (327.15,239.96) .. controls (327.15,239.69) and (326.92,239.46) .. (326.65,239.46) .. controls (326.37,239.46) and (326.15,239.69) .. (326.15,239.96) .. controls (326.15,240.24) and (326.37,240.46) .. (326.65,240.46) .. controls (326.92,240.46) and (327.15,240.24) .. (327.15,239.96) -- cycle ;
\draw  [fill={rgb, 255:red, 0; green, 0; blue, 0 }  ,fill opacity=1 ] (327.11,236.83) .. controls (327.11,236.56) and (326.88,236.33) .. (326.61,236.33) .. controls (326.33,236.33) and (326.11,236.56) .. (326.11,236.83) .. controls (326.11,237.11) and (326.33,237.33) .. (326.61,237.33) .. controls (326.88,237.33) and (327.11,237.11) .. (327.11,236.83) -- cycle ;
\draw  [fill={rgb, 255:red, 0; green, 0; blue, 0 }  ,fill opacity=1 ] (330.92,249.96) .. controls (330.92,249.69) and (330.7,249.46) .. (330.42,249.46) .. controls (330.15,249.46) and (329.92,249.69) .. (329.92,249.96) .. controls (329.92,250.24) and (330.15,250.46) .. (330.42,250.46) .. controls (330.7,250.46) and (330.92,250.24) .. (330.92,249.96) -- cycle ;
\draw  [fill={rgb, 255:red, 0; green, 0; blue, 0 }  ,fill opacity=1 ] (330.92,243.78) .. controls (330.92,243.51) and (330.7,243.28) .. (330.42,243.28) .. controls (330.15,243.28) and (329.92,243.51) .. (329.92,243.78) .. controls (329.92,244.06) and (330.15,244.28) .. (330.42,244.28) .. controls (330.7,244.28) and (330.92,244.06) .. (330.92,243.78) -- cycle ;
\draw  [fill={rgb, 255:red, 0; green, 0; blue, 0 }  ,fill opacity=1 ] (330.92,240.23) .. controls (330.92,239.95) and (330.7,239.73) .. (330.42,239.73) .. controls (330.15,239.73) and (329.92,239.95) .. (329.92,240.23) .. controls (329.92,240.5) and (330.15,240.73) .. (330.42,240.73) .. controls (330.7,240.73) and (330.92,240.5) .. (330.92,240.23) -- cycle ;
\draw  [fill={rgb, 255:red, 0; green, 0; blue, 0 }  ,fill opacity=1 ] (330.92,236.91) .. controls (330.92,236.64) and (330.7,236.41) .. (330.42,236.41) .. controls (330.15,236.41) and (329.92,236.64) .. (329.92,236.91) .. controls (329.92,237.19) and (330.15,237.41) .. (330.42,237.41) .. controls (330.7,237.41) and (330.92,237.19) .. (330.92,236.91) -- cycle ;
\draw  [fill={rgb, 255:red, 0; green, 0; blue, 0 }  ,fill opacity=1 ] (334.06,250.01) .. controls (334.06,249.73) and (333.83,249.51) .. (333.56,249.51) .. controls (333.28,249.51) and (333.06,249.73) .. (333.06,250.01) .. controls (333.06,250.28) and (333.28,250.51) .. (333.56,250.51) .. controls (333.83,250.51) and (334.06,250.28) .. (334.06,250.01) -- cycle ;
\draw  [fill={rgb, 255:red, 0; green, 0; blue, 0 }  ,fill opacity=1 ] (334.1,243.82) .. controls (334.1,243.55) and (333.87,243.32) .. (333.6,243.32) .. controls (333.32,243.32) and (333.1,243.55) .. (333.1,243.82) .. controls (333.1,244.1) and (333.32,244.32) .. (333.6,244.32) .. controls (333.87,244.32) and (334.1,244.1) .. (334.1,243.82) -- cycle ;
\draw  [fill={rgb, 255:red, 0; green, 0; blue, 0 }  ,fill opacity=1 ] (334.28,240.05) .. controls (334.28,239.77) and (334.05,239.55) .. (333.78,239.55) .. controls (333.5,239.55) and (333.28,239.77) .. (333.28,240.05) .. controls (333.28,240.32) and (333.5,240.55) .. (333.78,240.55) .. controls (334.05,240.55) and (334.28,240.32) .. (334.28,240.05) -- cycle ;
\draw  [fill={rgb, 255:red, 0; green, 0; blue, 0 }  ,fill opacity=1 ] (334.28,236.91) .. controls (334.28,236.64) and (334.05,236.41) .. (333.78,236.41) .. controls (333.5,236.41) and (333.28,236.64) .. (333.28,236.91) .. controls (333.28,237.19) and (333.5,237.41) .. (333.78,237.41) .. controls (334.05,237.41) and (334.28,237.19) .. (334.28,236.91) -- cycle ;
\draw  [fill={rgb, 255:red, 0; green, 0; blue, 0 }  ,fill opacity=1 ] (252.98,184.07) .. controls (253.17,183.87) and (253.16,183.56) .. (252.96,183.37) .. controls (252.76,183.18) and (252.45,183.18) .. (252.26,183.38) .. controls (252.06,183.58) and (252.07,183.9) .. (252.27,184.09) .. controls (252.47,184.28) and (252.79,184.27) .. (252.98,184.07) -- cycle ;
\draw  [fill={rgb, 255:red, 0; green, 0; blue, 0 }  ,fill opacity=1 ] (261.29,192.02) .. controls (261.48,191.82) and (261.48,191.5) .. (261.28,191.31) .. controls (261.08,191.12) and (260.76,191.13) .. (260.57,191.33) .. controls (260.38,191.53) and (260.39,191.84) .. (260.59,192.03) .. controls (260.78,192.22) and (261.1,192.22) .. (261.29,192.02) -- cycle ;
\draw  [fill={rgb, 255:red, 0; green, 0; blue, 0 }  ,fill opacity=1 ] (269.24,199.62) .. controls (269.44,199.42) and (269.43,199.1) .. (269.23,198.91) .. controls (269.03,198.72) and (268.71,198.73) .. (268.52,198.93) .. controls (268.33,199.13) and (268.34,199.44) .. (268.54,199.63) .. controls (268.74,199.82) and (269.05,199.82) .. (269.24,199.62) -- cycle ;
\draw  [color={rgb, 255:red, 208; green, 2; blue, 27 }  ,draw opacity=1 ] (362.83,122.63) -- (203.32,121.82) -- (203.9,7.35) ;

\draw (135.5,11) node [anchor=north west][inner sep=0.75pt]  [font=\footnotesize] [align=left] {$\displaystyle S_{0}$};
\draw (175.5,50.5) node [anchor=north west][inner sep=0.75pt]  [font=\footnotesize] [align=left] {$\displaystyle S_{1}$};
\draw (215.5,91) node [anchor=north west][inner sep=0.75pt]  [font=\footnotesize] [align=left] {$\displaystyle S_{2}$};
\draw (336.5,212) node [anchor=north west][inner sep=0.75pt]  [font=\footnotesize] [align=left] {$\displaystyle S_{pq-1}$};
\draw (371.5,44.9) node [anchor=north west][inner sep=0.75pt]    {$y$};
\draw (85.5,0.9) node [anchor=north west][inner sep=0.75pt]    {$z$};

\end{tikzpicture}

\end{center}

    \caption{The projection of {\rm Im}$(\psi)$ to the $y$-$z$ plane}
    \label{psi}
\end{figure}

For $0\leq r < pq$ let $m_r$ be the unique element of $M$ such that $m_r\equiv r\mod pq$ if there exists such an element, otherwise let $m_r:=r$. Define the function $\psi: \mathbb{N} \to \mathbb{R}^3$ as follows: let $n\in\N$, take the unique decomposition $n=n'+m_r$ such that $pq\mid n'$. For $n'<0$ let $\psi(n)=(n,0,0)$, for $n'=0$ let $\psi(m_r)=(m_r,1,1),$ otherwise let $\psi(n):= (n, (1- \frac{1}{g_1(n')+1})+r, (1- \frac{1}{g_2(n')+1})-r)$, where $g_1,g_2$ are the functions from the definition of $\phi$. Taking the projection of {\rm Im}$(\psi)$ to the $y$-$z$ plane, we can observe that the square $S_r=[r,r+1]\times [-r,-r+1]$ contains the image under $\psi$ of numbers that are congruent to $r$ modulo $pq$ and at least $m_r$, and the construction in $S_r$ is almost the same as in $S$, only it starts from $m_r$ instead of 0, see Figure \ref{psi}. Hence here we use that the elements of $M$ have different residues modulo $pq$.

Observe that every arithmetic progression with difference $d=p^iq^j$, $i, j>0$ and initial element $a_0=m_r+l\cdot p^iq^j$ (where $l\in\N$) is contained by $S_r$ for some $r$ on the $y$-$z$ plane. Similarly to the case $M=\{0\}$, the set $\{a_0 + k \cdot p^iq^j: k \in \mathbb{N},\,p^iq^j\mid (a_0-m_r)\}$ is exactly the inverse image under $\psi$ of the octant  $\{(x,y,z): x \geq a_0, y \geq 1-\frac{1}{i+1}+r, z \geq 1-\frac{1}{j+1}-r \}$, where $a_0\ge m_r$ and $m_r \equiv a_0 \mod pq$.

We have yet to consider the cases when the difference is $d\in\{1,p,q\}$. For $d=1,$ the image under $\psi$ of the arithmetic sequence $\{a_0+k: k\in\N\}$ is exactly $\mathrm{Im}(\psi)\cap \{(x,y,z): x\ge a_0, y\ge 0,z\ge -pq\}$. In case $d\in\{p,q\}$, every arithmetic progression $\{a_0+kd: k\in\N\}$ which has at least $\max(p,q)\cdot m_{\mathcal{T}_3}(k) $ elements  and an initial element $a_0=m_r+ld$ for some $l\in\N$
contains at least $m_{\mathcal{T}_3}(k)$ elements congruent to $m_r$ modulo $pq.$ These are elements of an arithmetic progression with difference $d=pq$ which we already considered in the previous case.

We prove $m_{\{p^iq^j:i,j \in \mathbb{N}\}, M}^{\infty}(k) = m_{\mathcal{T}_3}$ for the case $M=\{0\}$, as generalizing the proof to any $|M|=1$ is straightforward. We already showed $m_{\{p^iq^j:i,j \in \mathbb{N}\},\{0\}}^{\infty}(k) \le m_{\mathcal{T}_3}(k)$, to see the other direction, assign to every $T_A \in \mathcal{T}_3$ a hypergraph $A \in \mathcal{A}_{D,\{0\}}^{\infty}$, $D = \{p^iq^j:i,j \in \mathbb{N}\}$ by a function $\gamma: V(T_A) \to \mathbb{N}$ such that for every hyperedge (octant) $E\subseteq V(T_A)$ the vertex set (set of some natural numbers) $\gamma[E]$ is a hyperedge (arithmetic progression with proper difference) in $A$. We can assume that the vertices of $T_A$ have distinct $x$-, $y$- and also $z$-coordinates, since perturbing every point a little to obtain different coordinates results in a more extensive set of hyperedges. Order the elements of $V(T_A)$ in ascending order by their $x$-coordinate and denote by $x(P)$ the place of $P$ in the ordering for every $P \in V(T_A)$. Define $y(P)$ and $z(P)$ similarly. Let $\gamma(P) = p^{y(P)} \cdot q^{z(P)} \cdot k$ with some $k$ such that $\gamma(P)$ be the $x(P)^{th}$ largest value in {\rm Im}($\gamma$). Notice that $|V(T_A)|<\infty$. Now every vertex set in $T_A$ which is captured by an octant corresponds to a subset of an arithmetic progression with difference $p^iq^j$ and $a_0=0$.
\end{proof}

\begin{proof}[Proof of Theorem~\ref{vegtelen_sorozat}. (3)]
We use the fact that $m_{\mathcal{T}_4}(k) = \infty$ \cite[Theorem 12]{hextant}. Again, we assign to every $T_A \in \mathcal{T}_4$ a hypergraph $A \in \mathcal{A}_{D,\{0\}}^{\infty}$, $D = \{p^iq^j:i,j \in \mathbb{N}\}$ by a function $\eta: V(T_A) \to \mathbb{N}$ such that for every hyperedge (hextant) $E\subseteq T_A$ the vertex set (set of some natural numbers) $\eta[E]$ is a hyperedge (arithmetic progression with proper difference) in $A$. It follows that if $m := m_{\left\{ p_1^{j_1}p_2^{j_2}p_3^{j_3}:j_1,j_2,j_3 \in \mathbb{N}\right\},\{0\}}^{\infty}(k)$ would be finite then we could polichromaticly $k$-color the hyperedges of an arbitrary $(T_A)_{\geq m} \in \mathcal{T}_4$. 

Now we execute the very same method as in the end of the proof of Theorem 6.2. We can assume again that the vertices of $T_A$ have distinct $x$-, $y$-, $z$- and also $w$-coordinates. Order the elements of $V(T_A)$ in ascending order by their $x$-coordinate and denote by $x(P)$ the place of $P$ in the ordering for every $P \in V(T_A)$. Define $y(P)$, $z(P)$ and $w(P)$ similarly. Let $\eta(P) = p_1^{y(P)} \cdot p_2^{z(P)} \cdot p_3^{w(P)} \cdot k$ with some $k$ such that $\eta(P)$ be the $x(P)^{th}$ largest value in {\rm Im}($\eta$). Now every vertex set in $T_A$ captured by a hextant corresponds to a subset of an arithmetic progression with difference $p^iq^j$ and $a_0=0$. 
\end{proof}

\begin{proof} [Proof of Theorem~\ref{veges_sorozat}. (1)]
Firstly, we show that $m_{\{t^i:i \in \mathbb{N}\}, \{0\}}(k) \leq m_{\mathcal{B}}(k)$. To prove this, we construct a $\phi:\mathbb{N}\to \mathbb{R}^2$ injection such that for every fixed $i, j_0$ and $j_{\text{max}}$ the set $\phi[\{j \cdot t^i: j \in \mathbb{N}^+, j_0 \leq j \leq j_{\text{max}}\}]$ can be defined by a bottomless rectangle. Denote the arithmetic sequence $j_0 \cdot t^i, \ldots, j_{\text{max}} \cdot t^i$ by $a_0, \ldots, a_{\text{max}}$. If $n = j \cdot t^i$ where $t$ is not a divisor of $j$, let $t(n) = i$. Define $\phi(n)= (1- \frac{1}{n}, \frac{1}{t(n)})$ and $\phi(0)=(0,0)$. Notice that the vertical strip $\{(x,y): 1-\frac{1}{a} \leq x \leq 1-\frac{1}{b}\}$ for some $a \leq b \in \mathbb{N}$ contains the images of numbers between $a$ and $b$, and the half-plane $\{(x,y): y \leq \frac{1}{i}\}$ for some $i \in \mathbb{N}$ contains the images of numbers which are divisible by $t^i$. Therefore, the set $\phi[\{a_0, \ldots, a_{\text{max}}\}]$ is the intersection of {\rm Im}$(\phi)$ and the bottomless rectangle $\{(x,y): 1-\frac {1}{a_0} \leq x \leq 1- \frac {1}{a_{\text{max}}}, \text{ } y \leq \frac{1}{i}\}$ if $j_0>0$, and $\{(x,y): 0 \leq x \leq 1- \frac {1}{a_{\text{max}}}, \text{ } y \leq \frac{1}{i}\}$ if $j_0=a_0=0$.

Notice that in this construction {\rm Im}$(\phi)$ was contained in the bottomless rectangle $\{(x,y): 0 \leq x \leq 1, \text{ } y \leq 1\}$. In the general case, we place $t$ bottomless rectangles such that the bottomless rectangle $B_r$ ($0 \leq r < t$) will contain numbers congruent to $r$ modulo $t$. 

Define the function $\psi: \mathbb{N} \to \mathbb{R}^2$ such that for $n=m\cdot t+r$, $0 \leq r<t$, let $\psi(n):= (2r+1-\frac{1}{n}, \frac{1}{t(n)})$ and $\psi(0)=(0,0)$. Observe that every arithmetic progression with difference $d=t^i$, $i>0$ is contained by $B_r$ for some $r$, where $B_r = \{(x,y): 2r \leq x \leq 2r+1, \text{ } y \leq 1\}$, since the $t-$residue of such a sequence is constant. Similarly to the case $M=\{0\}$, the set $\{m_r + j \cdot t^i: j \in \mathbb{N}, j \leq j_{\text{max}}\}$ is the intersection of {\rm Im}$(\psi)$ and the bottomless rectangle $\{(x,y): 2r+1- \frac {1}{a_0} \leq x \leq 2r+1- \frac {1}{a_{\text{max}}}, \text{ } y \leq \frac{1}{i}\}$ where $r \equiv m_r \mod t$, $a_0 = m_r + 0 \cdot t^i -r, a_1 = m_r + 1 \cdot t^i- r, \ldots, a_{\text{max}} = m_r + j_{\text{max}} \cdot t^i -r$, if $m_r \neq r$ (thus $a_0 \neq 0$). If $m_r=r$, then the corresponding bottomless rectangle is $\{(x,y): 2r \leq x \leq 2r+1- \frac {1}{a_{\text{max}}}, \text{ } y \leq \frac{1}{i}\}$. 

In the case $i=0$ (and consequently $d=1$), every hyperedge with at least $|M|\cdot(m_{\mathcal{B}}(k) -1)+1$ vertices  
contain at least $m_{\mathcal{B}}(k)$ vertices in $B_r$ for some $r\in M$, according to the pigeonhole principle. The corresponding numbers of these vertices form an arithmetic progression with difference $d=t$. We have already proved that this implies containment in a bottomless rectangle.
\end{proof}

\begin{proof}[Proof of Theorem~\ref{veges_sorozat}. (2)]
Firstly, we define the hypergraph class $\mathcal{T}_{fin}$. Let $A \in \mathcal{T}_{fin}$ if $A$ is a hypergraph such that $V(A) \subset \mathbb{R}^3$, $|V(A)|< \infty$, and $$E(A) \subseteq \{V(A) \cap \{(x,y,z): x_1 \leq x \leq x_2, \text{ } y \leq y_0, \text { } z \leq z_0\}: \text{ } x_1,x_2,y_0,z_0 \in \mathbb{R}\}.$$
    \begin{lemma}
    $m_{\mathcal{T}_{fin}}(k) = \infty$.
    \end{lemma}
    \begin{proof}
        As stated previously, $m_{\mathcal{R}}(k) = \infty$ \cite[Theorem 3]{teglalap}.
        Take a plane $H_0$ in $\mathbb{R}^3$ whose normal vector is $(0,1,1)$, with coordinate axes given by $(1,0,0)$ and $(0,1,-1)$. Then for every axis-parallel rectangle $r$ of $H_0$ there exists a set of form $a=\{(x,y,z): x_1 \leq x \leq x_2, \text{ } y \leq y_0, \text { } z \leq z_0\}$ such that $a\cap H_0$ is exactly $r$, we refer to Figure \ref{ferde} for an illustration. Thus, for every $R \in \mathcal{R}$ on $H_0$, there exists a hypergraph $A \in \mathcal{T}_{fin}$, such that $V(R) = V(A)$, and every hyperedge in $R$ is also a hyperedge in $A$. Therefore, $m_{\mathcal{T}_{fin}}(k) \geq m_{\mathcal{R}}(k) = \infty$.
\end{proof}

\begin{figure} [h!]
    \centering
    \includegraphics[width=0.7\textwidth]{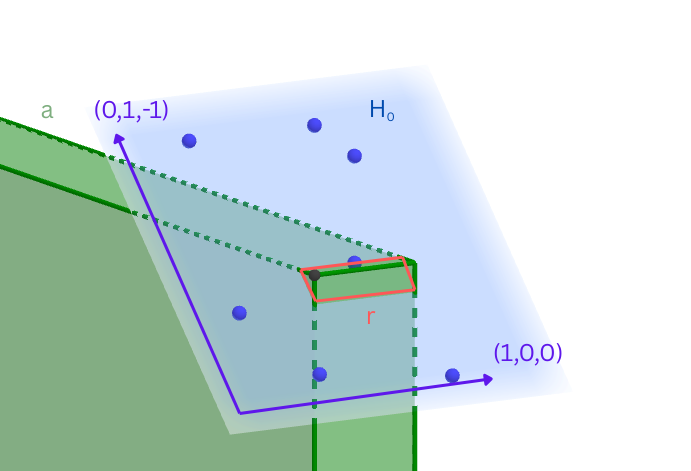}
    \caption{Cutting out an axis-parallel rectangle of $H_0$}
    \label{ferde}
\end{figure}

Our goal is to show that $m_{\{p^iq^j:i,j \in \mathbb{N}\},\{0\}}(k) \geq m_{\mathcal{T}_{fin}}(k)$. 

We have already shown in the proof of Theorem \ref{vegtelen_sorozat}. (2) that 
$m_{\{p^iq^j:i,j \in \mathbb{N}\},\{0\}}^{\infty}(k) \geq m_{\mathcal{T}_3}(k)$. To every $T_A \in \mathcal{T}_3$, we have assigned a hypergraph $A \in \mathcal{A}_{D,\{0\}}^{\infty}$, $D = \{p^iq^j:i,j \in \mathbb{N}\}$ by an injection $\gamma_A: V(T_A) \to \mathbb{N}$ such that for every hyperedge (octant) $E\subseteq V(T_A)$ the vertex set (set of some natural numbers) $\gamma_A[E]$ was a hyperedge (arithmetic progression with proper difference $d_E$) in $A$. Observe that every set $E_{x_1,x_2,y_0,z_0} = \{(x,y,z): x_1 \leq x \leq x_2, \text{ } y \leq y_0, \text { } z \leq z_0\}$ is the subset of an octant $O_{x_1,y_0,z_0} = \{(x,y,z): x_1 \leq x, \text{ } y \leq y_0, \text { } z \leq z_0\}$. 

Therefore, for every $T_{A'} \in \mathcal{T}_{fin}$ we have a $T_A \in \mathcal{T}_3$ such that $V(T_A) = V(T_{A'})$, and every hyperedge in $T_{A'}$ is the subset of a hyperedge of $T_A$. Use the same $\gamma_A$ injection to obtain a hypergraph $A' \in \mathcal{A}_{D,\{0\}}$. The image of every octant $O_{x_1,y_0,z_0} \cap V(T_A)$ is (the subset of) an arithmetic progresson with some $a_0$ and difference $d=p^iq^j$. Notice that $a_0$ is divisible by $d$. Take a hyperedge $E = E_{x_1,x_2,y_0,z_0} \cap V(T_{A'})$ of $T_{A'}$ with an arbitrary $x_2 \in \mathbb{R}$. The image $\gamma_A[E]$ is the prefix of the arithmetic progression above, with last element $a_n$. If $a_n = K \cdot d$ then $\gamma_A[E] = \{k \cdot d: k \in \mathbb{N}^+, k \leq K\} \cap V(A')$, which can be a hyperedge in $A'$.
Therefore, $m_{\{p^iq^j:i,j \in \mathbb{N}\},\{0\}}(k) \geq m_{\mathcal{T}_{fin}}(k) = \infty$.
\end{proof}

\section*{Acknowledgements}

The research was supported by the Lend\"ulet program of the Hungarian Academy of Sciences (MTA). BB was also supported by Lend\"ulet Grant no. 2022-58. of the Hungarian Academy of Sciences (MTA).

We are extremely grateful to Balázs Keszegh and Dömötör Pálvölgyi for their valuable remarks. We are also thankful to Sára Tóth for her collaboration throughout the research.

\bibliographystyle{plain}
\bibliography{ref}

@article{felsik,
  title={Polychromatic coloring for half-planes},
  author={Smorodinsky, S. and Yuditsky, Y.},
  journal={Journal of Combinatorial Theory, Series A},
  volume={119},
  number={1},
  pages={146--154},
  year={2012},
  publisher={Elsevier}
}

@article{feneketlen,
  title={Coloring hypergraphs induced by dynamic point sets and bottomless rectangles},
  author={Asinowski, A. and Cardinal, J. and Cohen, N. and Collette, S. and Hackl, T. and Hoffmann, M. and Knauer, K. and Langerman, S. and Laso{\'n}, M. and Micek, P. and Rote, G. and Ueckerdt, T.},
  journal={Algorithms and Data Structures, LNCS},
  volume={8037},
  pages={73--84},
  year={2013},
  organization={Springer}
}

@article{apstrip,
  title={Colorful strips},
  author={Aloupis, G. and Cardinal, J. and Collette, S. and Imahori, S. and Korman, M. and Langerman, S. and Schwartz, O. and Smorodinsky, S. and Taslakian, P.},
  journal={Graphs and combinatorics},
  volume={27},
  pages={327--339},
  year={2011},
  publisher={Springer}
}

@article{ternyolcad,
  title={More on decomposing coverings by octants},
  author={Keszegh, B. and P{\'a}lv{\"o}lgyi, D.},
  journal={Journal of Computational Geometry},
  volume={6},
  pages={300-315},
  year={2015}
}

@article{dekompozicio,
  title={Survey on Decomposition of Multiple Coverings},
  author={Pach, J. and P{\'a}lv{\"o}lgyi, D. and T{\'o}th, G.},
  journal={Bolyai Society Mathematical Studies},
  volume={24},
  pages={219-257},
  year={2014},
 publisher={Springer-Verlag}
}

@article{teglalap,
  title={Delaunay graphs of point sets in the plane with respect to axis-parallel rectangles},
  author={Chen, X. and Pach, J. and Szegedy, M. and Tardos, G.},
  journal={Random Structures \& Algorithms},
  volume={34},
  number={1},
  pages={11--23},
  year={2009}
}

@article{pshalfplane,
  title={An abstract approach to polychromatic coloring: shallow hitting sets in {A}{B}{A}-free hypergraphs and pseudohalfplanes},
  author={Keszegh, B. and P{\'a}lv{\"o}lgyi, D.},
  journal={Journal of Computational Geometry},
  volume={10},
  number={1},
  pages={1--26},
  year={2019}
}

@article{hextant,
  title={Proper coloring of geometric hypergraphs},
  author={Keszegh, B. and P{\'a}lv{\"o}lgyi, D.},
  journal={Discrete \& Computational Geometry},
  volume={62},
  number={3},
  pages={674--689},
  year={2019},
  publisher={Springer}
}

@article{ladder1,
  title={Monochromatic paths for the integers},
  author={Guerreiro, J. and Ruzsa, I. Z. and Silva, M.},
  journal={European Journal of Combinatorics},
  volume={58},
  pages={283--288},
  year={2016},
  publisher={Elsevier}
}

@article{ladder2,
  title={On the set of common differences in van der {Waerden}’s theorem on arithmetic progressions},
  author={Brown, T. C. and Graham, R. L. and Landman, B. M.},
  journal={Canadian Mathematical Bulletin},
  volume={42},
  number={1},
  pages={25--36},
  year={1999},
  publisher={Cambridge University Press}
}

@article{aw1,
  title={Rainbow arithmetic progressions},
  author={Butler, S. and Erickson, C. and Hogben, L. and Hogenson, K. and Kramer, L. and Kramer, R. and Lin, J. and Martin, R. and Stolee, D. and Warnberg, N. and Young, M.},
  journal={Journal of Combinatorics},
  volume={7},
  number={4},
  pages={595--626},
  year={2016}
}

@article{aw2,
  title={Anti-van der {Waerden} numbers of 3-term arithmetic progressions},
  author={Berikkyzy, Z. and Schulte, A. and Young, M.},
  journal={The Electronic Journal of Combinatorics},
  year={2017},
    volume={24},
    number={2}

}

@article{pach2016unsplittable,
  title={Unsplittable coverings in the plane},
  author={Pach, J. and P{\'a}lv{\"o}lgyi, D.},
  journal={Advances in Mathematics},
  volume={302},
  pages={433--457},
  year={2016},
  publisher={Elsevier}
}

@article{damasdi2021three,
  title={Three-chromatic geometric hypergraphs},
  author={Dam{\'a}sdi, G. and P{\'a}lv{\"o}lgyi, D.},
  journal={Journal of the European Mathematical Society},
    volume={published online first},
  year={2024}
}

@article{palvolgyi2010indecomposable,
  title={Indecomposable coverings with concave polygons},
  author={P{\'a}lv{\"o}lgyi, D.},
  journal={Discrete \& Computational Geometry},
  volume={44},
  pages={577--588},
  year={2010},
  publisher={Springer}
}

@article{palvolgyi2010convex,
  title={Convex polygons are cover-decomposable},
  author={P{\'a}lv{\"o}lgyi, D. and T{\'o}th, G.},
  journal={Discrete \& Computational Geometry},
  volume={43},
  number={3},
  pages={483--496},
  year={2010},
  publisher={Springer}
}

@article{varadarajan,
  title={Decomposing Coverings and the Planar Sensor Cover Problem},
  author={Gibson, M. and Varadarajan, K.},
  journal={Discrete \& Computational Geometry},
  volume={46},
  pages={313--333},
  year={2011},
  publisher={Springer}
}

@article{kovacs2015indecomposable,
  title={Indecomposable coverings with homothetic polygons},
  author={Kov{\'a}cs, I.},
  journal={Discrete \& Computational Geometry},
  volume={53},
  number={4},
  pages={817--824},
  year={2015},
  publisher={Springer}
}

@article{damasdi2022realizing,
  title={Realizing an m-uniform four-chromatic hypergraph with disks},
  author={Dam{\'a}sdi, G. and P{\'a}lv{\"o}lgyi, D.},
  journal={Combinatorica},
  volume={42},
  number={1},
  pages={1027--1048},
  year={2022},
  publisher={Springer}
}

@article{temavezetes,
title={personal communication},
author={Keszegh, B. and P{\'a}lv{\"o}lgyi, D.}

}

@article{keszegh2012octants,
  title={Octants are cover-decomposable},
  author={Keszegh, B. and P{\'a}lv{\"o}lgyi, D.},
  journal={Discrete \& Computational Geometry},
  volume={47},
  pages={598--609},
  year={2012},
  publisher={Springer}
}

@article{bless10,
    author = {Planken, T. and Ueckerdt, T.},
    title = {Polychromatic Colorings of Geometric Hypergraphs via Shallow Hitting Sets},
    journal = {arXiv preprint arXiv:2310.19982},
    year = {2023}
}

@article{cogezoo,
    title = {The geometric hypergraph zoo},
    journal = {https://coge.elte.hu/cogezoo.html}
}

@inproceedings{chekan2022polychromatic,
  title={Polychromatic Colorings of Unions of Geometric Hypergraphs},
  author={Chekan, V. and Ueckerdt, T.},
  booktitle={International Workshop on Graph-Theoretic Concepts in Computer Science},
  pages={144--157},
  year={2022},
  organization={Springer}
}

\medskip

\end{document}